\documentclass[10pt]{cmip-l}

\long\def\commentout#1{}

\commentout{
\usepackage{ifpdf}
\ifpdf
      \message{pdftrue}
\else
      \message{pdffalse}
\fi
}

\newif\ifprint% allow any tweaks needed to differentiate published version
\printfalse

\newif\ifarxiv% allow any tweaks needed to differentiate arXiv version
\arxivtrue

\newif\ifhyper
\hypertrue %hyperrefs turned on if true, off it false

\usepackage{xspace}
\usepackage{color}
\usepackage[pdftex]{graphicx}
\usepackage[permil]{overpic}
\usepackage{array}

\usepackage{indentfirst,pict2e}
\usepackage{amsfonts}
\usepackage[leqno]{amsmath}
\usepackage{amsthm}
\usepackage{amssymb}
\usepackage{latexsym}
\usepackage{amssymb}
\usepackage[cmtip,all]{xy}

\makeatletter
\let\cir\@undefined
\makeatother

\ifprint
	\definecolor{linkred}{rgb}{0,0,0} % black
	\definecolor{linkblue}{rgb}{0,0,0} % black
\else
	\definecolor{linkred}{rgb}{0.7,0.2,0.2}
	\definecolor{linkblue}{rgb}{0,0.2,0.6}
\fi

\ifhyper
	\usepackage{xr-hyper} % precedes hyper or else!
\else
	\usepackage{xr}
\fi

\ifhyper
	\usepackage[
		hyperindex,
		pagebackref,
		\ifarxiv\else pdftex,\fi%
		pdftitle={Effective divisors on moduli spaces},
		pdfauthor={Dawei Chen, Gavril Farkas and Ian Morrison},
		pdfdisplaydoctitle,
		pdfpagemode=UseNone,
		breaklinks=true,
		extension=pdf,
		bookmarks=false,
		plainpages=false,
		colorlinks,
		linkcolor=linkblue,
		citecolor=linkblue,
		urlcolor=linkred,
		pdfmenubar=true,
		pdftoolbar=true,
		pdfpagelabels,
		pdfpagelayout=SinglePage,
		pdfview=Fit,
		pdfstartview=Fit
	]{hyperref}
\fi

\newcommand{\net}[1]{{\texttt{#1}}}

\usepackage[backrefs,msc-links,nobysame]{amsrefs}

\renewcommand{\eprint}[1]{#1}

\makeatletter

\def\bib@div@mark#1{%
 \@mkboth{{#1}}{{#1}}%
	}
\def\print@backrefs#1{%
 \space\SentenceSpace$\leftarrow$\csname br@#1\endcsname
}
\catcode`\'=11 % letter
\renewcommand{\PrintAuthors}[1]{%
 \ifx\previous@primary\current@primary
  \sameauthors\@empty
 \else
  \def\current@bibfield{\bib'author}%
		  \PrintNames{}{}{\scshape #1}%
 \fi
}
\catcode`\'=12
\def\MRhref#1#2{%
 \begingroup
 \parse@MR#1 ()\@empty\@nil%
  \href{\MR@url}{\texttt{\@tempd\vphantom{()}}}%
  \ifx\@tempe\@empty
  \else
   \ \href{\MR@url}{\texttt{(\@tempe)}}%
  \fi
 \endgroup
}%
\def\MR#1{%
 \relax\ifhmode\unskip\spacefactor3000 \space\fi
 \begingroup
 \strip@MRprefix#1\@nil
  \edef\@tempa{\@nx\MRhref{MR\@tempa}{\@tempa}}%
 \@xp\endgroup
 \@tempa
}
\makeatother

\definecolor{authornote}{rgb}{0,0.6,0} % green
\ifprint
	\def\authorsnote#1{}
\else
	\def\authorsnote#1{{\color{authornote}#1}}
\fi

\copyrightinfo{2012}{American Mathematical Society}
\newtheorem{theorem}[equation]{Theorem}
\newtheorem{lemma}[equation]{Lemma}
\newtheorem{corollary}[equation]{Corollary}
\newtheorem{conjecture}[equation]{Conjecture}

\theoremstyle{definition}
\newtheorem{definition}[equation]{Definition}
\newtheorem{example}[equation]{Example}
\newtheorem{problem}[equation]{Problem}

\theoremstyle{remark}
\newtheorem{remark}[equation]{Remark}

\numberwithin{equation}{section}

\def\normallinespread{1.15}
\def\normalspread{\linespread{\normallinespread}\normalfont\selectfont}
\def\bibreflinespread{0.98}
\def\bibrefspread{\linespread{\bibreflinespread}\normalfont\selectfont}
\newcommand{\ds}[1]{\displaystyle{#1}}
\newcommand{\PP}{\ensuremath{\mathbb{P}}}

\newcommand{\irr}{\operatorname{irr}}
\newcommand{\HH}{\ensuremath{\mathcal{H}}}
\newcommand{\A}{\ensuremath{\mathcal{A}}}
\newcommand{\RR}{\ensuremath{\mathbb{R}}}
\newcommand{\TT}{\ensuremath{\mathcal{T}}}
\newcommand{\tp}{\operatorname{top}}

\newcommand{\dblq}{{/{\kern-3pt}/}} 
\def\OO{\mathcal{O}}

\def\cA{\mathcal{A}}
\def\F{\mathcal{F}}

\def\L{\mathcal{L}}

\def\aa{\overline{\mathcal{A}}}

\newcommand{\C}{\mathbb{C}}

\newcommand{\Q}{\mathbb{Q}}
\newcommand{\G}{\mathbb{G}}

\newcommand{\M}{\ensuremath{\mathcal M}}

\newcommand{\Mbar}{\ensuremath{\overline{\M}}}

\newcommand{\m}{\ensuremath{M}}

\newcommand{\mbar}{\ensuremath{\overline{\m}}}

\newcommand{\mijbar}[2]{\ensuremath{\mbar_{{#1},{#2}}}}

\newcommand{\deltairr}{\ensuremath{\delta_{\irr}}} 
\newcommand{\Deltairr}{\ensuremath{\Delta_{\irr}}}

\newcommand{\moobarMM}[3]{\ensuremath{\mijbar{0}{0}(#1,#2,#3)}} 

\newcommand{\thst}[2]{\ensuremath{{#1}^{\mathrm{#2}}}}

\newcommand{\eff}[1]{\mathrm{Eff}\bigl({#1}\bigr)}
\newcommand{\nef}[1]{\mathrm{Nef}\bigl({#1}\bigr)}

\newcommand{\Si}[1]{{\mathfrak S}_{#1}}

\newcommand{\Sn}{\Si{n}}

\newcommand{\Dsym}{\widetilde{\Delta}}

\newcommand{\pic}[1]{\mathrm{Pic}\bigl({#1}\bigr)}
\newcommand{\picone}{\mathrm{Pic}^{\underline{1}}}
\newcommand{\picdeg}[2]{\mathrm{Pic}^{#1}\bigl({#2}\bigr)}
\newcommand{\sspan}{\mathrm{span}}
\newcommand{\conv}{\mathrm{Conv}}
\newcommand{\Proj}{\mathrm{Proj}}
\newcommand{\Chow}{\mathrm{Chow}}
\newcommand{\capac}{\mathrm{cap}}
\newcommand{\PGL}{\PP\mathrm{GL}}
\newcommand{\shavedast}{\ast \kern -1.25pt}

\newcommand{\ooo}{\ensuremath{\mathcal{O}}} % structure sheaf
\newcommand{\Deltawt}{\Delta_{\text{wt}}}
\newcommand{\bwt}{b_{\text{wt}}}

\newcommand{\nnn}{{<}n{>}}
\newcommand{\ddd}{{<}d{>}}

\begin{document}
	
% try to get better float placement
\renewcommand{\topfraction}{0.85}
\renewcommand{\textfraction}{0.1}
\renewcommand{\floatpagefraction}{0.85}
\normalspread

\title{Effective divisors on moduli spaces}

\author{Dawei Chen}
\address{Department of Mathematics\\ Boston College\\ Chestnut Hill, MA 02467, USA}
\email{dawei.chen@bc.edu}
\thanks{}

\author{Gavril Farkas}
\address{Humboldt Universit\"at zu Berlin\\Institut f\"ur Mathematik\\ Unter den Linden 6\newline \indent  Berlin 10099, Germany}
\email{farkas@math.hu-berlin.de}
\thanks{}

\author{Ian Morrison}
\address{Department of Mathematics\\ Fordham University\\ Bronx, NY 10458, USA}
\email{morrison@fordham.edu}

\date{\today}

\thanks{The first author is partially supported by the NSF grant DMS-1200329.}
\dedicatory{\vskip8pt \normalsize Dedicated to Joe Harris---master geometer, inspired teacher and valued friend---\newline on the occasion of his \thst{\textit{60}}{\textit{th}} birthday.\vskip8pt}

%    Body of article.

%\begin{abstract}
%We survey what is known about cones of effective divisors on moduli spaces of stable curves and maps, focusing on the minimal slope $s_g$ of classes on $\mgbar$ lying in the $\lambda-\delta$ plane. The first section reviews constructions for such classes and the upper bounds for $s_g$ that these give and the second deals with methods for producing lower bounds for $s_g$. In the third section, we collect results and conjectures for effective cones of other moduli spaces, focussing on genus $0$.
%\end{abstract}

\maketitle
\thispagestyle{empty}% suppress page number on title pagesr

%    Body of article.

\section*{Introduction}

The pseudo-effective cone $\mbox{Eff}(X)$ of a smooth projective variety $X$ is a fundamental, yet elusive invariant. On one hand, a few general facts are known: the interior of the effective cone is the cone of big divisors so, in particular, $X$ is of general type if and only if $K_X\in \mbox{int}(\mbox{Eff}(X))$;
less obviously \cite{BDPP}, a variety $X$ is uniruled if and only if $K_X$ is not pseudo-effective and the dual of $\mbox{Eff}(X)$ is the cone of movable curves; and, the effective cone is known to be polyhedral for Fano varieties. For further background, see \cite{LazarsfeldPositivityI}.
On the other hand, no general structure theorem is known and the calculation of $\mbox{Eff}(X)$ is a daunting task even in some of the simplest cases. For instance, the problem of computing the cone $\mbox{Eff}(C^{(2)})$ for a very general curve $C$ of genus $g$ is known to be equivalent to Nagata's Conjecture, see \cite{CilibertoKouvidakis}.

The aim of this paper is to survey what is known about the effective cones of moduli spaces, with a focus on the moduli spaces $\Mbar_{g,n}$ of stable curves, $\aa_g$ of  principally polarized abelian varieties and $\Mbar_{g, n}(X, \beta)$ of stable maps. Because related moduli spaces often have an inductive combinatorial structure and the associated families provide a rich cycle theory, the study of effective cones of moduli spaces has often proven more tractable and more applicable than that of general algebraic varieties. 

For example, in the case of $\Mbar_g$, we may define, following~\cite{HarrisMorrisonSlopes}, the slope $s(D)$ of a divisor class $D$ of the form $a\lambda - b\delta -c_{\irr}
 \deltairr -  \sum_{i} c_i\delta_i$, with $a$ and $b$ positive, all the $c$'s non-negative and at least one---in practice, almost always $c_{\irr}$---equal to $0$, to be $\frac{a}{b}$. We set $s(D) = \infty$ for divisors not of this form, for example, if $g\ge 3$, for the components $\Delta_{\irr}$ and $\Delta_i$. A fundamental invariant is then the slope $s(\Mbar_g):=\mbox{inf}\{ s(D) \,|\, D \in \mbox{Eff}(\Mbar_g)\}$.
The Harris-Mumford theorem \cite{HarrisMumfordKodaira} on the Kodaira dimension of $\Mbar_g$, is equivalent to the inequality $s(\Mbar_g)<s(K_{\Mbar_g})=\frac{13}{2}$, for $g\geq 24$. For a long time, the conjecture of \cite{HarrisMorrisonSlopes} that the inequality $s(\Mbar_g)\geq 6+\frac{12}{g+1}$ holds, equality being attained only for the classical Brill-Noether divisors whose classes were also computed in~\cite{HarrisMumfordKodaira}, was widely believed.  Counterexamples were provided in \cite{FarkasKoszul} for infinitely many $g$ though all of these have slope greater than $6$. On the other hand, all the methods (cf. \cites{HarrisMorrisonSlopes,ChenRigid, PandharipandeSlope}) for bounding $s(\Mbar_g)$ from below for large $g$, yield only bounds that tend to zero with $g$. This sets the stage for the following fundamental question:
\begin{problem}
Does the limit $\ds{s_{\infty}:=\lim_{g\to  \infty} s(\Mbar_g)}$ exist, and, if so, what is it's value?
\end{problem}
The authors know of no credible, generally accepted conjectural answer. The first tends to guess that $s_{\infty}=0$, the second and third that $s_{\infty} = 6$. Hedging his guess, the third author has a dinner \emph{bet} with the second, made at the 2009 MSRI Program in Algebraic Geometry: the former wins if $s_{\infty}=0$, the latter if $s_{\infty}>0$, and the bet is annulled should the limit not exist.

The argument for $s_{\infty}=0$ is that the papers cited above. which compute the invariants of movable curves in
$\Mbar_g$ using tools as diverse tools as Hurwitz theory, Teichm\"uller dynamics and Hodge integrals, do no better than $s(\Mbar_g)\geq O(\frac{1}{g})$. Intriguingly, the first two methods, though apparently quite different in character, suggest the \emph{same} heuristic lower bound $\frac{576}{5g}$ for the slope; see Section 3 of this paper. Is this coincidence or evidence for the refined asymptotic ${\liminf_{g\to  \infty} g\,s(\Mbar_g) = \frac{576}{5}}$ conjectured by the first author in \cite{ChenRigid}, and hence that $s_{\infty}=0$?

The argument for $s_{\infty}>0$ is that effective divisors of small slope are known to have strong geometric characterizations: for instance, they must contain the locus $\mathcal{K}_g$ of curves lying on $K3$ surfaces. Constructing \emph{any} such divisors, let alone ones of arbitrarily small slope, is notoriously difficult. In fact, for $g\geq 11$, not a single example of an effective divisor having slope less than $6+\frac{10}{g}$ is known. The current state of knowledge concerning divisors of small slope is summarized in Section 2 of the paper.

We invite the reader to take sides in this bet, or much better, settle it conclusively by computing $s_{\infty}$. To encourage work that might enable him to win, the third author here announces the First Morrison Prize, in the amount of US\$100, for the construction of any effective divisor on $\Mbar_g$ of slope less than $6$, as determined by a jury consisting of the present authors.
One further question is to what extent $s_{\infty}$ has a modular meaning. As pointed out in \cite{HarrisMorrisonSlopes}*{p. 323}, the inequality $s_{\infty}>0$ would imply a fundamental difference between the geometry of $\M_g$ and $\cA_g$ and provide a new geometric approach to the Schottky problem.

\vskip 3pt
We now describe the contents of the paper. Section 1 recalls the classical constructions of effective divisors on $\Mbar_g$, starting with Brill-Noether and Gieseker-Petri divisors.  Then we discuss the cases $g\leq 9$, where a much better understanding of the effective cone is available and alternative Mukai models of $\Mbar_g$ are known to exist. In Section 2, we highlight the role of syzygy divisors in producing examples of divisors on $\Mbar_g$ of small slope and discuss the link to an interesting conjecture of Mercat \cite{Mercat} that suggests a stratification of $\M_g$ in terms of rank $2$ vector bundles on curves. Special attention is paid to the interesting transition case $g=11$, which is treated from the point of view both of Koszul cohomology and higher rank Brill-Noether theory.

Section 3 is devoted to finding lower bounds on $s(\Mbar_g)$ and the existing methods are surveyed. The common idea is to find a Zariski dense collection of $1$-cycles $B_{\mu}$, so that 
any effective divisor must intersect one of these curves non-negatively, obtaining the bound $s(\Mbar_g) \ge \inf_{\mu} \Bigl(\frac{B_{\mu}\cdot \delta}{B_{\mu}\cdot \lambda}\Bigr)$. There are several methods of constructing these curves, e.g. by using simply-branched coverings of $\PP^1$ and allowing a pair of branch points to come together \cite{HarrisMorrisonSlopes}, by imposing conditions on curves in projective spaces, especially canonical space \cites{CoskunHarrisStarrEffective, FedorchukSeveri}, as Teichm\"uller curves arising from branched covers of elliptic curves \cite{ChenRigid}, or as complete intersection of nef tautological divisors on $\Mbar_g$, with intersection numbers evaluated via Gromov-Witten theory \cite{PandharipandeSlope}. 

In Section 4, we turn to moduli of abelian varieties and discuss the recent paper \cite{FGSMV} showing that the Andreotti-Mayer divisor $N_0'$ of $5$-dimensional ppav whose theta divisor is singular at a pair of points which are not two-torsion computes the slope of the perfect cone compactification $\aa_5$ of $\cA_5$ as $s(\aa_5)=\frac{54}{7}$. 

Section 5 is devoted almost exclusively to moduli spaces of curves of genus $g=0$.  We begin with a few cases---the space $\widetilde{\M}_{0, n}$ that is the quotient of $\Mbar_{0, n}$ by the natural action of $\Sn$ induced by permuting the marked points and the Kontsevich moduli spaces of stable maps $\Mbar_{0, 0}(\PP^d, d)$---in which unpublished arguments of Keel make it easy to determine the effective cone completely.
We then discuss more systematically the space $\Mbar_{0, 0}(\PP^d, d)$, sketching the sharper results of Coskun, Harris and Starr~\cite{CoskunHarrisStarrEffective} on their effective cones. We also review some of the results of the first author with Coskun and Crissman concerning the Mori program for these spaces, emphasizing the examples $\Mbar_{0, 0}(\mathbb P^3, 3)$ where~\cite{ChenLogMinimal} completely works out the geometry of this program, giving an explicit chamber decomposition of the effective cone in terms of stable base loci, and $\Mbar_{0, 0}(\PP^4, 4)$ for which much, though not all, the geometry is worked out in~\cite{ChenCoskunCrissman}.

The rest of Section 5 deals with results for $\Mbar_{0, n}$. For $n\leq 5$, the naive guess that the effective cone might be generated by the components of the boundary is correct, and we recall the argument for this. But for larger $n$ new extremal rays appear. We first review the example of Keel and Vermeire~\cite{Vermeire} and the proof of Hassett and Tschinkel~\cite{HassettTschinkelEffective} that, for $n=6$, there are no others. The main focus of this section is to give a brief guide to the ideas of Castravet and Tevelev~\cite{CastravetTevelev} which show just how rapidly the complexity of these effective cones grows.

We conclude this introduction by citing some work on effective divisors that we have not reviewed. These include Rulla's extensions in~\cites{RullaThesis, RullaEffective} of the ideas in \S\S\ref{symmetric} to quotients by subgroups permuting only a subset of the marked points and Jensen's examples~\cite{JensenFibrations} for $\Mbar_{5, 1}$ and $\Mbar_{6, 1}$. In a very recent preprint, Cooper~\cite{CooperQuotients} studies the moduli spaces of stable quotients $\overline{Q}_{1,0}(\PP^{n-1},d)$ of Marian, Oprea and Pandharipande \cite{MarianOpreaPandharipande}. Because there is a surjection $\Mbar_{g, n}(\PP^{n-1},d) \to \overline{Q}_{g,0}(\PP^{n-1},d)$, this is relevant to \S\S\ref{moobarrd}. In the case of $g=1$ and $n=0$ that Cooper considers, the target  $\overline{Q}_{1,0}(\PP^{n-1},d)$ is smooth with a 
rank~$2$ Picard group and she is able to describe the effective (and nef) cones explicitly. Finally, we have not touched upon connections with the $F$-conjecture, including Pixton's exciting example \cite{PixtonExample} of an effective divisor on $\Mbar_{0, 12}$ that intersects all topological $1$-strata non-negatively yet is not equivalent to an effective sum of boundary divisors.

\medskip
\noindent\textbf{Conventions and notation} To simplify notation, we will ignore torsion classes and henceforth use $\pic{M}$ with no decoration for $\pic{M} \otimes \Q$. We set $\eff{M}$ and $\nef{M}$ for the effective and nef cones of $M$. We denote by $\mbox{Mov}(M)$ the cone of movable divisors on $M$ parametrizing effective divisors whose stable base locus has codimension at least $2$ in $M$. We write $\deltairr$ for the class of the boundary component of irreducible nodal curves, and, when there is no risk of confusion, we simplify notation by omitting the limits of summations indexed by boundary components consisting of reducible curves. We work throughout over $\C$.

\section{Geometric divisors on $\Mbar_g$}
\label{Lowerbounds}

\commentout{
Following \cite{HarrisMorrisonSlopes}, we recall the definition of the slope of the moduli space of curves. If $\delta:=\deltairr+\cdots +\delta_{[g/2]}$ denotes the class of the total
boundary, first one defines the slope of an effective divisor $D$ on $\overline{\mathcal{M}}_g$ by the formula
$$s(D):= \mbox{inf }\Bigl\{\frac{a}{b}:a,b>0 \mbox{ such that }
a\lambda-b\delta-[D]= \sum_{j=0}^{\lfloor \frac{g}{2}\rfloor} c_j\delta_j,\mbox{ where
}c_j\geq 0\Bigr\}.$$ From the definition it follows that $s(D)=\infty$
unless $[D]=a\lambda-\sum_{i=0}^{\lfloor \frac{g}{2}\rfloor} b_i\delta_i$ with
$a,b_i\geq 0$, for all $i$. This is the case when  $D$ contains no boundary divisors in its support, therefore in this case one has that
$$s(D)=\frac{a}{\mathrm{min}_{i=0}^{{\lfloor \frac{g}{2}\rfloor }} b_i}\geq 0.$$  For $g\geq 3$, it is known that the boundary classes $\deltairr, \ldots, \delta_{[g/2]}$ are linearly independent in $\mbox{Pic}(\Mbar_g)$, therefore the definition implies that $s(\Delta_i)=\infty$, for $0\leq i\leq \lfloor \frac{g}{2}\rfloor$. We define the
\emph{slope of the moduli space $\Mbar_g$} as the quantity
$$s(\Mbar_g):=\mbox{inf
}\bigl\{s(D):D\in \mbox{Eff}(\Mbar_g)\bigr\}.$$
}

Any expression $D = a\lambda - b_{\irr}\deltairr -  \sum_{i} b_i\delta_i$ for an effective divisor $D$ on $\Mbar_g$ (with all coefficients positive) provides an upper bound for $s(\Mbar_g)$. Chronologically, the first such calculations are those of the Brill-Noether divisors, which we briefly recall following \cites{HarrisMumfordKodaira,EisenbudHarrisKodaira}.
\begin{definition}
For positive integers $g, r, d\geq 1$ such that $$\rho(g, r, d):=g-(r+1)(g-d+r)=-1,$$ we denote by $\M_{g, d}^r:=\{[C]\in \M_g: W^r_d(C)\neq \emptyset\}$ the Brill-Noether locus of curves carrying a linear series of type $\mathfrak g^r_d$.
\end{definition}
It is known \cite{EisenbudHarrisLimitBNDivisors} that $\M_{g, d}^r$ is an irreducible effective divisor. The class of its closure in $\Mbar_g$ has been computed in
\cite{EisenbudHarrisKodaira} and one has the formula
$$[\Mbar_{g, d}^r]=c_{g, r, d}\Bigl((g+3)\lambda-\frac{g+1}{6}\deltairr-\sum_i i(g-i)\delta_i\Bigr),$$
where $c_{g, r, d}\in \mathbb Q_{>0}$ is an explicit constant that can be viewed as an intersection number of Schubert cycles in a Grassmannian. Note that $s(\Mbar_{g, d}^r)=6+\frac{12}{g+1}$, thus implying  the upper bound
$s(\Mbar_g)\leq 6+\frac{12}{g+1}$, for all $g$ such that $g+1$ is composite, so that the diophantine equation $\rho(g, r, d)=-1$ has integer solutions.
The initial Slope Conjecture \cite{HarrisMorrisonSlopes} predicted that the Brill-Noether divisors are divisors of minimal slope. This turns out to be true only when $g\leq 9$ and $g=11$.

Observe that remarkably, for various $r, d\geq 1$ such that $\rho(g, r, d)=-1$, the classes of the divisors $\Mbar_{g, d}^r$ are proportional. The proof given in \cite{EisenbudHarrisKodaira} uses essential properties of Picard groups of moduli spaces of pointed
curves and it remains a challenge to find an \emph{explicit} rational equivalence linking the various Brill-Noether divisors on $\Mbar_g$. The first interesting case is $g=11$, when there are two Brill-Noether divisors, namely $\Mbar_{11, 6}^1$ and $\Mbar_{11, 9}^2$. Note that when $g=2$, the divisor $\Delta_1$ has the smallest slope $10$ in view of the relation  $10\lambda = \deltairr + 2\delta_1$ on $\Mbar_2$, see for instance \cite{Moduli}*{Exercise (3.143)}.

When $g=3, 5, 7, 8, 9, 11$, there exist Brill-Noether divisors which actually determine the slope $s(\Mbar_g)$. This has been shown in a series of papers \cites{HarrisMorrisonSlopes,ChangRanKodaira,TanSlopes,FarkasPopa} in the last two decades. Some cases have been recovered recently in \cites{CoskunHarrisStarrEffective, FedorchukSeveri}. 

For $3 \leq g \leq 9$ and $g = 11$, it is well known \cite{Mukai1} that a general curve of genus $g$ can be realized as a hyperplane section $H$ of a $K3$ surface $S$ of degree $2g-2$ in $\PP^g$. Consider a general Lefschetz pencil $B$ in the linear system $|H|$. Blowing up the $2g-2$ base points of $B$, we get a fibration $S'$ over $B$, with general fiber a smooth genus $g$ curve. All singular fibers are irreducible one-nodal curves.
From the relation
$$\chi_{\tp}(S') = \chi_{\tp}(B)\cdot \chi_{\tp}(F) + \mbox{the number of nodal fibers}, $$
where $F$ is a smooth genus $g$ curve, we conclude that
$$B\cdot \deltairr = 6g + 18, \quad B\cdot \delta_i = 0 \ \mbox{for}\ i > 0. $$
Let $\omega$ be the first Chern class of the relative dualizing sheaf of $S'$ over $B$. By the relation
$$12\lambda = \delta + \omega^2$$
and
$$\omega^2 = c^2_1(S') + 4(2g-2) = 6g-6,$$
we obtain that
$$B\cdot \lambda = g+1. $$
Consequently the slope of the curve $B$ is given by
$$ s_B = 6 + \frac{12}{g+1}. $$
Since the pencil $B$ fills-up $\Mbar_g$ for $g\leq 9$ or $g=11$, we get the lower bound $s(\Mbar_g)\geq 6+\frac{12}{g+1}$ in this range.  The striking  coincidence  between the slope of the Brill-Noether divisors $\Mbar_{g, d}^r$ and that of Lefschetz pencil on a fixed $K3$ surface of genus $g$ has a transparent explanation in view of Lazarsfeld's result \cite{LazarsfeldBN}, asserting that every nodal curve $C$ lying on a $K3$ surface $S$ such that $\mbox{Pic}(S)=\mathbb Z[C]$, satisfies the Brill-Noether theorem, that is, $W^r_d(C)=\emptyset$ when $\rho(g, r, d)<0$. In particular, when $\rho(g, r, d)=-1$, the intersection
of the pencil $B\subset \Mbar_g$ with the Brill-Noether divisor $\Mbar_{g, d}^r$ is \emph{empty}, therefore also, $B\cdot \Mbar_{g, d}^r=0$. This confirms the formula
$$s(\Mbar_{g, d}^r)=s_B=6+\frac{12}{g+1}.$$
This Lefschetz pencil calculation also shows \cite{FarkasPopa} that any effective divisor $D\in \mbox{Eff}(\Mbar_g)$ such that $s(D)<6+\frac{12}{g+1}$ must necessarily contain the locus
$$\mathcal{K}_g:=\{[C]\in \mathcal{M}_g: C\mbox{ lies on a}\ K3\ \mbox{surface}\}.$$
In particular, effective divisors of slope smaller than $6+\frac{12}{g+1}$ have a strong geometric characterization, hence constructing them is relatively difficult.  If one views a divisor on $\Mbar_g$ as being given in terms of a geometric condition that holds in codimension one in moduli, then in order for such a condition to lead to a divisor of small slope on $\Mbar_g$, one must search for geometric properties that single out sections of $K3$ surfaces among all curves of given genus. Very few such geometric properties are known at the moment, for curves on $K3$ surfaces are known to behave generically with respect to most geometric stratifications of $\M_g$, for instance those given by gonality or existence of special Weierstrass points.

For integers $g$ such that $g+1$ is prime, various substitutes for the Brill-Noether divisors have been proposed, starting with the \emph{Gieseker-Petri} divisors. Recall that the Petri Theorem asserts that for a line bundle  $L$ on a general curve $C$ of genus $g$, the multiplication map
$$\mu_0(L):H^0(C, L)\otimes H^0(C, K_C\otimes L^{\vee})\rightarrow H^0(C, K_C)$$
is injective. This implies that the scheme $G^r_d(C)$ classifying linear series of type $\mathfrak g^r_d$ is smooth of expected dimension $\rho(g, r, d)$ when $C$ is general. The first proof of this statement was given by Gieseker whose argument was later greatly simplified in \cite{EisenbudHarrisPetri}. Eventually, Lazarsfeld \cite{LazarsfeldBN} gave the most elegant proof, and his approach has the added benefit of singling out curves on very general K3 surfaces as the only collections of smooth curves of arbitrary genus verifying the Petri condition. The locus where the Gieseker-Petri theorem does not hold is the proper subvariety of the moduli space
$$\mathcal{GP}_g:=\{[C]\in \M_g: \mu_0(L) \mbox{ is not injective for a certain}\ L\in \mbox{Pic}(C)\}\,.$$
This breaks into subloci $\mathcal{GP}_{g, d}^r$ whose general point corresponds to a curve $C$ such that $\mu_0(L)$ is not injective for some linear series $L\in W^r_d(C)$. The relative position of the subvarieties $\mathcal{GP}_{g, d}^r$ is not yet well-understood. The following elegant prediction was communicated to the second author by Sernesi:
\begin{conjecture}\label{sernesi}
The locus $\mathcal{GP}_g$ is pure of codimension one in $\M_g$.
\end{conjecture}
Clearly there are loci $\mathcal{GP}_{g, d}^r$ of codimension higher than one. However, in light of Conjecture \ref{sernesi} they should be contained in other Petri loci in $\M_g$ that fill-up a codimension one component in moduli. Various partial results in this sense are known. Lelli-Chiesa \cite{Lelli} has verified Conjecture~\ref{sernesi} for all $g\leq 13$. It is proved in \cite{FarkasRatlmaps} that whenever $\rho(g, r, d)\geq 0$, the locus $\mathcal{GP}_{g, d}^r$ carries at least a divisorial component. Bruno and Sernesi \cite{BS} show that $\mathcal{GP}_{g, d}^r$ is pure of codimension one for relatively small values of $\rho(g, d, r)$, precisely
$$0<\rho(g, r, d)<g-d+2r+2.$$ The problem of computing the class of the closure $\overline{\mathcal{GP}}_{g, d}^r$  has been completely solved only when the Brill-Noether numbers is equal to $0$ or $1$. We quote from \cite{EisenbudHarrisKodaira} (for the case $r=1$) and \cite{FarkasKoszul} (for the general case $r\geq 1$).
\begin{theorem}
Fix integers $r, s\geq 1$ and set $d:=rs+r$ and $g:=rs+s$, therefore $\rho(g, r, d)=0$. The slope of the corresponding Gieseker-Petri divisor is given by the formula:
$$s(\overline{\mathcal{GP}}_{g, d}^r)=6+ \frac{12}{g+1}+\frac{6(s+r+1)(rs+s-2)(rs+s-1)}{s(s+1)(r+1)(r+2)(rs+s+4)(rs+s+1)}.$$
\end{theorem}

For small genus, one recovers the class of the divisor $\mathcal{GP}_{4, 3}^1$ of curves of genus $4$ whose canonical model lies on a quadric cone and then $s(\overline{\mathcal{GP}}_{4, 3}^1)=\frac{17}{2}$. When $g=6$, the locus $\mathcal{GP}_{6, 4}^1$ consists of curves  whose canonical model lies on a singular del Pezzo quintic surface and then $s(\overline{\mathcal{GP}}_{6, 4}^1)=\frac{47}{6}$. In both cases, the Gieseker-Petri divisors attain the slope of the respective moduli space. 

We briefly recall a few other divisor class calculations. For 
genus $g=2k$, Harris has computed in \cite{HarrisKodaira} the class of the divisor $\mathfrak{D}_1$ whose general point corresponds to a curve $[C]\in \M_g$ having a pencil $A\in W^1_{k+1}(C)$ and a point $p\in C$  with $H^0(C, A(-3p))\neq 0$. This led to the first proof that $\Mbar_g$ is of general type for even $g\geq 40$. This was superseded in \cite{EisenbudHarrisKodaira}, where with the help of Gieseker-Petri and Brill-Noether divisors, it is proved that $\Mbar_g$ is of general type for all $g\geq 24$.

Keeping $g=2k$, if $\sigma: \overline{\mathcal{H}}_{g, k+1}\rightarrow \Mbar_g$ denotes the generically finite forgetful map from the space of admissible covers of genus $g$ and degree $k+1$, then $\mathfrak{D}_1$ is the push-forward under $\sigma$ of a boundary divisor on $\overline{\mathcal{H}}_{g, k+1}$, for the general point of the Hurwitz scheme corresponds to a covering with \emph{simple} ramification. The other divisor appearing as a push-forward under $\sigma$ of a boundary locus in $\overline{\mathcal{H}}_{g, k+1}$ is the divisor $\mathfrak{D}_2$ with general point corresponding to a curve $[C]\in \M_g$ with a pencil $A\in W^1_{k+1}(C)$ and two points $p, q\in C$ such that $H^0(C, A(-2p-2q))\neq 0$. The class of this divisor has been recently computed by van der Geer and Kouvidakis~\cite{VdGKouvidakis2}.

An interesting aspect of the geometry of the Brill-Noether divisors is that for small genus, they  are rigid, that is, $[\Mbar_{g, d}^r]\notin \mbox{Mov}(\Mbar_g)$, see for instance \cite{FarkasRatlmaps}. This is usually proved by exhibiting a  curve $B\subset \Mbar_{g, d}^r$ sweeping out $\Mbar_{g, d}^r$ such that $B\cdot \Mbar_{g, d}^r<0$. Independently of this observation, one may consider the slope
$$s'(\Mbar_g):=\mbox{inf}\{s(D): D\in \mbox{Mov}(\Mbar_g)\}$$ of the cone of movable divisors. For $g\leq 9$, the inequality $s'(\Mbar_g)>s(\Mbar_g)$ holds.

\subsection{Birational models of $\Mbar_g$ for small genus.}
We discuss models of $\Mbar_g$ in some low genus cases, when this space is unirational (even rational for $g\leq 6$) and one has a better understanding of the chamber decomposition of the effective cone.
\begin{example}
We set $g=3$ and let $B\subset \Mbar_3$ denote the family induced by a pencil of curves of type $(2, 4)$ on $\mathbb P^1\times \mathbb P^1$. All members in this family are hyperelliptic curves. A standard calculation gives that $B\cdot \lambda=3$ and $B\cdot \deltairr=28$, in particular $B\cdot \Mbar_{3, 2}^1=-1$. This implies not only that the hyperelliptic divisor $\Mbar_{3, 2}^1$ is rigid, but also the inequality $s'(\Mbar_3)\geq s_B=\frac{28}{3}$. This bound is attained via the birational  map
$$\varphi_3:\Mbar_3\dashrightarrow X_3:=|\mathcal{O}_{\mathbb P^2}(4)|\dblq SL(3)$$ to the GIT quotient of plane quartics. Since $\varphi_3$ contracts the hyperelliptic divisor $\Mbar_{3, 2}^1$ to the point corresponding to double conics, from the push-pull formula one finds that  $s(\varphi_3^*(\mathcal{O}_{X_3}(1))=\frac{28}{3}$.
This proves the equality $s'(\Mbar_3)=\frac{28}{3}>9=s(\Mbar_3).$
\end{example}
That $s'(\Mbar_g)$ is accounted for by a rational map from $\Mbar_g$ to an \emph{alternative moduli space of curves} of genus $g$, also holds for a few higher genera, even though the geometry quickly becomes intricate.

\begin{example}
For the case $g=4$, we refer to \cite{Fedorchuk4}. Precisely, we introduce the moduli space $X_4$ of
$(3, 3)$ curves on $\mathbb P^1\times \mathbb P^1$, that is, the GIT quotient
$$X_4:=|\mathcal{O}_{\mathbb P^1\times \mathbb P^1}(3, 3)|\dblq SL(2)\times SL(2).$$
There is a birational map $\varphi:\Mbar_4\dashrightarrow X_4$, mapping an abstract genus $4$ curve $C$ to $\mathbb P^1\times \mathbb P^1$ via the two linear series $\mathfrak g^1_3$ on $C$. The Gieseker-Petri divisor is contracted to \emph{the point} corresponding to triple conics. This shows that $[\overline{\mathcal{GP}}_{4, 3}^1]\in \mbox{Eff}(\Mbar_4)$ is an extremal point. By a local analysis, Fedorchuk computes in \cite{Fedorchuk4} that $s(\varphi_4^*(\mathcal{O}(1,1)))=\frac{60}{9}=s'(\Mbar_4)>s(\Mbar_4)$. Furthermore, the model $X_4$ is one of the log-canonical models of $\Mbar_4$.
\end{example}

Mukai \cites{Mukai1,Mukai2,MukaiFano} has shown that general canonical curves of genus $g=7, 8, 9$ are linear sections of a rational homogeneous variety $$V_g\subset \mathbb P^{\mathrm{dim}(V_g)+g-2}.$$ This construction induces a new model $X_g$ of $\Mbar_g$ having Picard number equal to $1$, together with a birational map $\varphi_g:\Mbar_g\dashrightarrow X_g$. Remarkably, $s(\varphi_g^*(\mathcal{O}_{X_g}(1))=s'(\Mbar_g)$. The simplest case is $g=8$, which we briefly explain. 

\begin{example} Let $V:=\mathbb C^6$ and consider $\mathbb G:=G(2, V)\subset \mathbb P(\bigwedge^2 V)$. Codimension $7$ linear sections of $\mathbb G$ are canonical curves of genus $8$, and there is a birational map
$$\varphi_8:\Mbar_8\dashrightarrow X_8:=G(8, \bigwedge^2 V)\dblq SL(V)\,,$$
that is shown in \cite{Mukai2} to admit a beautiful interpretation in terms of rank two Brill-Noether theory.
The map $\varphi_8^{-1}$ associates to a general projective $7$-plane $H\subset \mathbb P(\bigwedge^2 V)$ the curve $[\mathbb G\cap H]\in \M_8$. In particular,  a smooth curve $C$ of genus $8$ appears as a linear section of $\mathbb G$ if and only if $W^2_7(C)=\emptyset$.
Observing that $\rho(X_8)=1$, one expects that exactly five divisors get contracted under $\varphi_8$, and indeed---see \cites{FarkasVerraodd, FarkasVerraNik}---
 $$\mbox{Exc}(\varphi_8)=\{\Delta_1, \Delta_2, \Delta_3, \Delta_4, \Mbar_{8, 7}^2\}\,.$$
Using the explicit construction of $\varphi_8$ one can show that the Brill-Noether divisor
gets contracted to a point. Thus $X_8$ can be regarded as a (possibly simpler) model of $\Mbar_8$ in which plane septimics are excluded.
\end{example}

\subsection{Upper bounds on the slope of the moving cone}
If $f:X\dashrightarrow Y$ is a rational map between normal projective varieties, then $f^*(\mbox{Ample}(Y))\subset \mbox{Mov}(X)$. In order to get upper bounds on $s'(\Mbar_g)$ for arbitrary genus, a logical approach is to consider rational maps from $\Mbar_g$ to other projective varieties and compute pull-backs of ample divisors from the target variety.
Unfortunately there are only few known examples of such maps, but recently two examples have been worked out. We begin with \cite{FarkasRatlmaps}, where a map between two moduli spaces of curves is considered.

We fix an odd genus $g:=2a+1\geq 3$ and set $g':=\frac{a}{a+1}{2a+2\choose a}+1$. Since
$\rho(2a+1, 1, a+2)=1$, we can define a rational map $\phi_a: \Mbar_{g} \dashrightarrow \Mbar_{g'}$ that associates to a curve $C$ its \emph{Brill-Noether curve} $ \phi([C]):=[W^1_{a+2}(C)]$ consisting of pencils of minimal degree---that the genus of $W^1_{a+2}(C)$ is is $g'$ follow from the Harris-Tu formula for Chern numbers of kernel bundles, as explained in \cite{EisenbudHarrisKodaira}.   Note that $\phi_1:\Mbar_3\dashrightarrow \Mbar_3$ is the identity map, whereas the map $\phi_2:\Mbar_5\dashrightarrow \Mbar_{11}$ has a rich and multifaceted geometry.
For a general $[C]\in \M_5$, the Brill-Noether curve $W_4^1(C)$ is endowed with a fixed point free involution $\iota:L\mapsto K_C\otimes L^{\vee}$. The quotient curve 
$\Gamma:=W^1_4(C)/\iota$ is a smooth plane quintic which can be identified with the space of singular quadrics containing the canonical image $C\hookrightarrow \mathbb P^4$. 
Furthermore, Clemens showed that the Prym variety induced by $\iota$ is precisely the Jacobian of $C$! This result has been recently generalized by Ortega \cite{Ortega} to all odd genera.
Instead of having an involution, the curve $W^1_{a+2}(C)$ is endowed with a fixed point free correspondence 
$$\Sigma:=\Bigl\{(L, L'): H^0(L')\otimes H^0(K_C\otimes L^{\vee})\rightarrow 
H^0(K_C\otimes L'\otimes L^{\vee}) \mbox { is not injective}\Bigr\},$$
which induces a \emph{Prym-Tyurin variety} $P\subset \mbox{Jac}(W^1_{a+2}(C)$ of exponent equal to the Catalan number $\frac{(2a)!}{a! (a+1)!}$ and $P$ is isomorphic to the Jacobian of the original curve $C$.

The main result of \cite{FarkasRatlmaps} is a complete description of the pull-back map $\phi_a^*$ at the level of divisors implying the slope evaluation:
\begin{theorem}\label{slopepull} For any
divisor class $D\in \mathrm{Pic}(\Mbar_{g'})$ having slope $s(D)=s$,\\[-7pt] 
$$s(\phi_a^*(D))=6+ \frac{8a^3(s-4)+5sa^2-30a^2+20a-8as-2s+24}{a(a+2)(sa^2-4a^2-a-s+6)}\ .$$
\end{theorem}
By letting $s$ become very large, one obtains the estimate $s'(\Mbar_g)<6+\frac{16}{g-1}$.

A different approach is used by van der Geer and Kouvidakis \cite{VdGKouvidakis1} in even genus $g=2k$. We consider once more the Hurwitz scheme
$\sigma:\overline{\mathcal{H}}_{g, k+1}\rightarrow \Mbar_g$. Associate to a degree $k+1$ covering   $f:C\rightarrow \mathbb P^1$ the \emph{trace curve}
$$T_{C, f}:=\{(x, y)\in C\times C: f(x)=f(y)\}.$$ For a generic choice of $C$ and $f$, the curve $T_{C, f}$ is smooth of genus $g':=5k^2-4k+1$. By working in families one obtains
a rational map $\chi:\overline{\mathcal{H}}_{g, k+1}\dashrightarrow \Mbar_{g'}$. Observe that as opposed to of the map $\phi_a$ from \cite{FarkasRatlmaps}, the ratio
$\frac{g'}{g}$ for the genera of the trace curve and that of the original curve is much lower. The map $\sigma_* \chi^*:\mbox{Pic}(\Mbar_{g'})\rightarrow \mbox{Pic}(\Mbar_g)$ is completely described in \cite{VdGKouvidakis2} and the estimate
$$s'(\Mbar_g)<6+\frac{18}{g+2}$$
is shown to hold for all even genera $g$. In conclusion, $\Mbar_g$ carries moving divisors of slope $6+O\bigl(\frac{1}{g}\bigr)$ for any genus. We close by posing the following question:

\begin{problem} Is it true that $\liminf_{g\rightarrow \infty} s(\Mbar_g)=\liminf_{g\rightarrow \infty} s'(\Mbar_g)$?
\end{problem}

\section{Syzygies of curves and upper bounds on $s(\Mbar_g)$}
The best known upper bounds on $s(\Mbar_g)$ are given by the Koszul divisors of \cites{FarkasSyzygies,FarkasKoszul} defined in terms of curves having unexpected syzygies. An extensive survey of this material, including an alternative proof using syzygies of the Harris-Mumford theorem \cite{HarrisMumfordKodaira} on the Kodaira dimension of $\Mbar_g$ for odd genus $g>23$,  has appeared in \cite{Farkassdg}.  Here we shall be brief and concentrate on the latest developments.
 
As pointed out in \cite{FarkasPopa} as well as earlier in this survey, any effective divisor $D\in \mbox{Eff}(\Mbar_g)$ of slope $s(D)<6+\frac{12}{g+1}$ must necessarily contain the locus $\mathcal{K}_g$ of curves lying on $K3$ surfaces. It has been known at least since the work of Mukai \cite{MukaiProc} and Voisin \cite{VoisinActa} that a curve $C$ lying on a  $K3$ surface
$S$ carries special linear series that are not projectively normal. For instance, if $A\in W^1_{\lfloor{\frac{g+3}{2}\rfloor}}(C)$ is a pencil of minimal degree, then the multiplication 
map for the residual linear system
$$\mbox{Sym}^2 H^0(C, K_C\otimes A^{\vee})\rightarrow H^0(C, K_C^{\otimes 2}\otimes A^{\otimes (-2)})$$
is not surjective. One can interpret projective normality as being the Green-Lazarsfeld property $(N_0)$ and accordingly, stratify $\M_g$ with strata consisting of curves $C$  that fail the higher properties $(N_p)$ for $p\geq 1$, for a certain linear system $L\in W^r_d(C)$ with $h^1(C, L)\geq 2$. This stratification of $\M_g$ is fundamentally different from classical stratifications given in terms
 of gonality or Weierstrass points (for instance the Arbarello stratification). In this case, the locus $\mathcal{K}_g$ lies in the smallest stratum, that is, it plays the role of the hyperelliptic locus $\M_{g, 2}^1$ in the gonality stratification! Observe however, that this idea, when applied to the canonical bundle $K_C$ (when of course $h^1(C, K_C)=1$), produces exactly the gonality stratification, see \cites{Farkassdg,Green} for details. Whenever the second largest stratum in the new Koszul stratification is of codimension $1$, it will certainly contain $\mathcal{K}_g$ and is thus a good candidate for being a divisor of small slope. The main difficulty in carrying out this program lies not so much in computing the virtual classes of the Koszul loci, but in proving that they are divisors when one expects them to be so.

We begin by recalling basic definitions and refer to the beautiful book of Aprodu-Nagel \cite{AproduNagel} for a geometrically oriented introduction to syzygies on curves.
\begin{definition}
For a smooth curve $C$, a line bundle $L$ and a sheaf $\F$ on $C$, we define the Koszul cohomology group $K_{p, q}(C;\F, L)$ as the cohomology of the complex
$$\bigwedge^{p+1} H^0(C, L)\otimes H^0(C, \F\otimes L^{\otimes(q-1)})\stackrel{d_{p+1, q-1}}\longrightarrow \bigwedge^p H^0(C, L)\otimes H^0(C, \F\otimes L^{\otimes q})\stackrel{d_{p, q}}\longrightarrow$$\nopagebreak
$$ \stackrel{d_{p, q}}\longrightarrow \bigwedge^{p-1} H^0(C, L)\otimes H^0(C, \F\otimes L^{\otimes (q+1)}).$$
\end{definition}
It is a basic fact of homological algebra that the groups $K_{p, q}(C; \F, L)$ describe the graded pieces of the minimal resolution of the graded ring $$R(\F, L):=\bigoplus_{q\geq 0} H^0(C, \F\otimes L^{\otimes q})$$ as an $S:=\mbox{Sym } H^0(C, L)$-module.
Precisely, if $F_{\bullet}\rightarrow R(\F, L)$ denotes the minimal graded free resolution with graded pieces $F_p=\oplus_q S(-q)^\oplus {b_{pq}}$, then 
$$\mbox{dim}\ K_{p, q}(C; \F, L)=b_{pq}, \mbox{ for all }p, q\geq 0.$$ 
When $\F=\OO_C$, one writes $K_{p, q}(C, L):=K_{p, q}(C; \OO_C, L)$. 

\begin{example} Green's Conjecture \cite{Green} concerning the  syzygies of a  canonically embedded  curve $C\hookrightarrow \mathbb P^{g-1}$ can be formulated as an equivalence
$$K_{p, 2}(C, K_C)=0\Leftrightarrow p<\mbox{Cliff}(C).$$
Despite a lot of progress, the conjecture is still wide open for arbitrary curves. Voisin has proved the conjecture for general curves of arbitrary genus  in \cites{VoisinEven, VoisinOdd}. In odd genus $g=2p+3$, the conjecture asserts that the resolution of a general curve $[C]\in \M_{2p+3}$ is pure and has precisely the form:
$$0\rightarrow S(-g-1)\rightarrow S(-g+1)^{\oplus b_1}\rightarrow \cdots \rightarrow S(-p-3)^{\oplus b_{p+1}}\rightarrow S(-p-1)^{\oplus b_p}\rightarrow \cdots
$$
$$
\rightarrow S(-2)^{\oplus b_3}\rightarrow S(-2)^{\oplus b_1}\rightarrow R(K_C)\rightarrow 0.$$
The purity of the generic resolution in odd genus is reflected in the fact that the syzygy jumping locus 
$$\{[C]\in \M_{2p+3}: K_{p, 2}(C, K_C)\neq 0\}$$ is a virtual divisor, that is, a degeneracy locus
between vector bundles of the same rank over $\M_{2p+3}$. It is the content of Green's Conjecture that set-theoretically, this virtual divisor is an honest divisor which moreover coincides with the Brill-Noether divisor
$\M_{g, p+2}^1$.
\end{example}  
One defines a \emph{Koszul locus} on the moduli space as the subvariety consisting of curves $[C]\in \M_g$ such that $K_{p, 2}(C, L)\neq 0$, for a certain special linear system $L\in W^r_d(C)$.
The case when $\rho(g, r, d)=0$ is treated in  the papers \cite{FarkasSyzygies} and \cite{FarkasKoszul}. For the sake of comparison with the case of positive Brill-Noether number, we quote a single result, in the simplest case  $p=0$, when the syzygy condition $K_{0, 2}(C, L)\neq 0$ is equivalent to requiring that the embedded curve $C\stackrel{|L|}\to \mathbb P^r$ lie on a quadric hypersurface.
\begin{theorem}\label{rho0}
Fix $s\geq 2$ and set $g=s(2s+1)$, $r=2s$ and $d=2s(s+1)$. The locus in moduli 
$$\mathcal{Z}_s:=\bigl\{[C]\in \M_g: K_{0, 2}(C, L)\neq 0 \mbox{ for a certain}\ L\in W^r_d(C)\bigr\}$$
is an effective divisor on $\M_g$. The slope of its closure in $\Mbar_g$ is equal to  
$$s(\overline{\mathcal{Z}}_{s})=\frac{a}{b_0}=\frac{3(16s^7-16s^6+12s^5-24s^4-4s^3+41s^2+9s+2)}{s
(8s^6-8s^5-2s^4+s^2+11s+2)}.$$
\end{theorem}
This implies that $s(\overline{\mathcal{Z}}_{s})<6+\frac{12}{g+1}$. In particular $s(\Mbar_g)<6+\frac{12}{g+1}$, for all genera of the form $g=s(2s+1)$. In the case $s=2$, one has the set-theoretic equality of divisors $\mathcal{Z}_2=\mathcal{K}_{10}$ and $s(\overline{\mathcal{Z}}_{2})=s(\Mbar_{10})=7$. This was the first instance of a geometrically defined divisor on $\Mbar_g$ having smaller slope than that of the Brill-Noether divisors, see \cite{FarkasPopa}.  

The proof of Theorem \ref{rho0} breaks into two parts, very different in flavor. First one computes the virtual class of $\overline{\mathcal{Z}}_s$, which would then equal the actual class $[\overline{\mathcal{Z}}_s]$, if one knew that
$\mathcal{Z}_s$ was a divisor on $\M_g$. This first step has been carried out independently and with different techniques by the second author in \cite{FarkasKoszul} and by Khosla in \cite{Khosla}. The second step in the proof involves showing  that $\mathcal{Z}_s$ is a divisor. It suffices to exhibit a single curve $[C]\in \M_g$ such that $K_{0, 2}(C, L)=0$, for every linear series $L\in W^r_d(C)$.
By a standard monodromy argument, in the case $\rho(g,r,d)=0$, this is equivalent to the seemingly weaker requirement that there exist both a curve $[C]\in \M_g$ and a \emph{single}
linear series $L\in W^r_d(C)$ such that $K_{0, 2}(C, K_C)=0$. This is proved by degeneration in \cite{FarkasKoszul}.

The case of Koszul divisors defined in terms of linear systems with positive Brill-Noether number is considerably more involved, but the rewards are also higher. For instance, this approach is used in \cite{Farkassdg} to prove that $\Mbar_{22}$ is of general type. 

We fix integers $s\geq 2$ and $a\geq 0$, then set
$$g=2s^2+s+a, \quad d=2s^2+2s+a,$$
 therefore $\rho(g, r, d)=a$. Consider the stack $\sigma:\mathfrak G^r_d\rightarrow \M_g$ classifying linear series  $\mathfrak g^r_d$ on curves of genus $g$. Inside the stack $\mathfrak{G}^r_d$ we consider the locus of those pairs $[C, L]$ with $L\in W^r_d(C)$, for which the multiplication map
$$\mu_0(L): \mathrm{Sym}^2 H^0(C, L)\rightarrow H^0(C, L^{\otimes 2})$$ is not injective. The expected codimension in $\mathfrak{G}^r_d$ of this cycle is equal to $a+1$, hence the push-forward under $\sigma$ of this cycle is a virtual divisor in $\M_g$. The case $a=1$ of this construction will be treated in the forthcoming paper \cite{FarkasKoszul2}
from which we quote:
\begin{theorem}\label{nr1}
We fix $s\geq 2$ and set $g=2s^2+s+1$. The locus
$$\mathfrak D_s:=\{[C]\in \M_g: K_{0, 2}(C, L)\neq 0 \ \mbox{for a certain}\ L\in W^{2s}_{2s(s+1)+1}(C)\}$$
is an effective divisor on $\M_g$. The slope of its closure inside $\Mbar_g$ equals
$$s(\overline{\mathfrak D}_s)=\frac{3(48s^8-56s^7+92s^6-90s^5+86s^4+324s^3+317s^2+182s+48)}{24s^8-28s^7+22s^6-5s^5+43s^4+112s^3+100s^2+50s+12}.
$$
\end{theorem}
We observe that the inequality $$6+\frac{10}{g}<s(\overline{\mathfrak{D}}_s)<6+\frac{12}{g+1},$$
holds for each $s\geq 3$. The case $s=3$ of Theorem \ref{nr1} is presented in \cite{Farkassdg} and it proves that 
 $\Mbar_{22}$ is a variety of general type. Another very interesting case is $s=2$, that is, $g=11$. This case is studied by Ortega and the second author in
\cite{FarkasOrtega} in connection with \emph{Mercat's Conjecture} in higher rank Brill-Noether theory. In view of the relevance of this case to attempts of establishing a 
credible rank two Brill-Noether theory, we briefly explain the situation.

We denote as usual by $\F_g$ the moduli space parametrizing pairs $[S, H]$, where $S$ is a smooth $K3$ surface and $H\in \mathrm{Pic}(S)$ is a primitive nef line bundle with $H^2=2g-2$. 
Over $\F_g$ one considers the projective bundle
$\mathcal{P}_g$
classifying pairs $[S, C]$, where  $S$ is a smooth  $K3$  surface and $C\subset S$ is a smooth curve of genus $g$. Clearly $\mbox{dim}(\mathcal{P}_g)=\mbox{dim}(\F_g)+g=19+g$.
Observe now that for $g=11$ both spaces $\M_{11}$ and $\mathcal{P}_{11}$ have the same dimension, so one expects a general curve of genus $11$ to lie on finitely many $K3$ surfaces. This expectation can be made much more precise. 

For a general curve $[C]\in \M_{11}$, the rank $2$ Brill-Noether locus
$$\mathcal{SU}_C(2, K_C, 7):=\{E\in \mathcal{U}_C(2, 20):\mbox{det}(E)=K_C, \ h^0(C, E)\geq 7\}$$
is a smooth $K3$ surface. Mukai shows in \cite{Mukaigenus11} that $C$ lies on a unique $K3$ surface which can be realized as the Fourier-Mukai partner of $SU_C(2, K_C, 7)$. Moreover, this procedure induces a birational isomorphism
$$\phi_{11}:\M_{11}\dashrightarrow \mathcal{P}_{11},\  \ \mbox{    } \phi_{11}([C]):=\bigl[\widehat{\mathcal{SU}_C(2, K_C, 7)}, C\bigr]$$
and the two Brill-Noether divisors $\Mbar_{11, 6}^1$ and $\Mbar_{11,9}^2$ (and likewise the Koszul divisor) are pull-backs by $\phi$ of  Noether-Lefschetz divisors on $\F_{11}$. 

Next we define the second Clifford index of a curve which measures the complexity of a curve in its moduli space from the point of view of rank two vector bundles.
\begin{definition}
If $E\in \mathcal{U}_C(2, d)$ denotes a semistable vector bundle of rank $2$ and degree $d$ on a curve $C$ of genus $g$, one defines its Clifford index as
$\gamma(E):=\mu(E)- h^0(C, E)+2$ and the \emph{second Clifford index} of $C$ by the quantity
$$\mathrm{Cliff}_2(C):=\mbox{min}\bigl\{\gamma(E): E\in \mathcal{U}_C(2, d), \ \ d\leq 2(g-1), \ \ h^0(C, E)\geq 4\bigr\}.$$
\end{definition}

Mercat's Conjecture \cite{Mercat} predicts that the equality 
\begin{equation}\label{mercat}
\mbox{Cliff}_2(C)=\mbox{Cliff}(C)
\end{equation}
holds for \emph{every} smooth curve of genus $g$. By specializing to direct sums of line bundles, the inequality $\mbox{Cliff}_2(C)\leq \mbox{Cliff}(C)$ is obvious. Lange and Newstead have proved the conjecture for small genus \cite{LangeNewstead}. However the situation changes for $g=11$ and the following result
is proved in \cite{FarkasOrtega}:

\begin{theorem}\label{m11}
The Koszul divisor $\mathfrak{D}_{2}$ on $\M_{11}$ has the following realizations:
\begin{enumerate}
\item (By definition) $\{[C]\in \M_{11}: \exists L\in W^4_{13}(C)\ 
\mbox{such that} \ K_{0, 2}(C, L)\neq 0 \}$.
\item $\{[C]\in \M_{11}: \mathrm{Cliff}_2(C)<\mathrm{Cliff}(C)\}.$
\item $\phi_{11}^*(\mathcal{NL})$, where $\mathcal{NL}$ is the Noether-Lefschetz divisor of (elliptic) $K3$ surfaces $S$ with lattice $\mathrm{Pic}(S)\supset \mathbb Z\cdot H\oplus \mathbb Z\cdot C$, where $C^2=20$, $H^2=6$ and $C\cdot H=13$.
\end{enumerate}     
The closure $\overline{\mathfrak{D}}_2$ of $\mathfrak{D}_2$ in $\Mbar_{11}$ is a divisor of minimal slope
$$s(\overline{\mathfrak{D}}_2)=s(\Mbar_{11})=7.$$
\end{theorem}

In \cite{FarkasOrtega}, the second description in Theorem \ref{m11} is shown to imply that $\mathfrak{D}_2$ is the locus over which Mercat's Conjecture fails. Since $\mathfrak{D}_2\neq \emptyset$, Mercat's Conjecture in its original form is false on $\M_{11}$. On the other hand, Theorem \ref{m11} proves equality (\ref{mercat}) for general curves $[C]\in \M_{11}$. Proving Mercat's Conjecture
for a general $[C]\in \M_g$, or understanding the loci in moduli where the equality fails, is a stimulating open question.

\begin{remark}
Note that in contrast with lower genus, for $g=11$ we have $s(\Mbar_{11})=s'(\Mbar_{11})=7$. Furthermore, the dimension of the linear system 
of effective divisors of slope $7$ is equal to $19$ (see \cite{FarkasPopa}). The divisors $\Mbar_{11, 6}^1$, $\Mbar_{11, 9}^2$ and $\overline{\mathfrak{D}}_2$ are just three 
elements of this $19$-dimensional linear system.
\end{remark}

\section{Lower bounds on $s(\Mbar_g)$}
In this section we summarize several approaches towards finding lower bounds for $s(\Mbar_g)$ when $g$ is large. The idea is simple. One has to produce one-dimensional collections of families $C \to B$ of genus $g$ curves---in other words, curves $B \in \Mbar_g$---such that the union of all the curves $B$ is Zariski dense. For example, a single moving curve $B$ provides such a collection. No effective divisor $D$ contain all such $B$ and, when $B$ does not lie in $D$, the inequality  $B\cdot D \geq 0$ implies that the \emph{slope} $ s_B := \frac{B\cdot \delta}{B\cdot \lambda}$ is a lower bound for $s(D)$, and hence that the infimum of these slopes is a lower bound for $s(\Mbar_g)$. The difficulty generally arises in computing this bound. We discuss several constructions via covers of the projective line, via imposing conditions on space curves, via Teichm\"uller theory and via Gromov-Witten theory, respectively. Observe that when $\Mbar_g$ is a variety of general type, it carries no rational or elliptic moving curves.

\subsection{Covers of $\PP^1$ with a moving branch point}
\label{HarrisMorrison}

Harris and the third author \cite{HarrisMorrisonSlopes} constructed moving curves in $\Mbar_g$ using certain Hurwitz curves of branched covers of
$\PP^1$. Consider a connected $k$-sheeted cover $f: C\to \PP^1$ with $b = 2g - 2 + 2k$ simply branch points $p_1, \ldots, p_b$. If $\gamma_i$ is a closed loop around $p_i$ separating it from the other $p_j$'s, then, since $p_i$ is a simple branch point, the monodromy around $\gamma_i$ is a simple transposition $\tau_i$ in the symmetric group $S_k$ on the points of a general fiber. The product of these transpositions (suitable ordered) must be the identity since a loop around all the $p_i$ is nullhomotopic, the subgroup they generate must be transitive (since we assume that $C$ is connected), and then covers with given branch points are specified by giving the $\tau_i$ up to simultaneous $S_k$-conjugation.

Varying $p_b$ while leaving the others fixed, we obtain a one-dimensional Hurwitz space $Z$. When $p_b$ meets another branch point, say $p_{b-1}$, the base $\PP^1$ degenerates to a union of two $\PP^1$-components glued at a node $s$, with
$p_1, \ldots, p_{b-2}$ on one component and $p_{b-1}$ and $p_b$ on the other. The covering curve $C$ degenerates to a nodal \emph{admissible cover} accordingly, with nodes, say, $t_1, \ldots, t_n$. Such covers were introduced by Beauville for $k=2$ \cite{BeauvillePrym} and by Harris and Mumford \cite{HarrisMumfordKodaira} for general $k$; for a non-technical introduction to admissible covers, see \cite{Moduli}*{Chapter 3.G}.

Locally around $t_i$, the covering map is given by $(x_i, y_i) \to (u = x_i^{k_i}, v = y_i^{k_i})$, where $x_i, y_i$ and $u, v$ parameterize the two branches meeting at $t_i$ and $s$, respectively. We still call $k_i - 1$ the order of ramification of $t_i$. The data $(n; k_1, \ldots, k_n)$ can be determined by the monodromy of the cover around $p_{b-1}$ and $p_b$. 

When $p_b$ approaches $p_{b-1}$, consider the product $\tau = \tau_{b-1}\tau_b$ associated to the vanishing cycle $\beta$ that shrinks to the node $s$ as shown in Figure~\ref{degenerate}.

\begin{figure}[ht]
    \centering
	\begin{overpic}[scale=0.5]{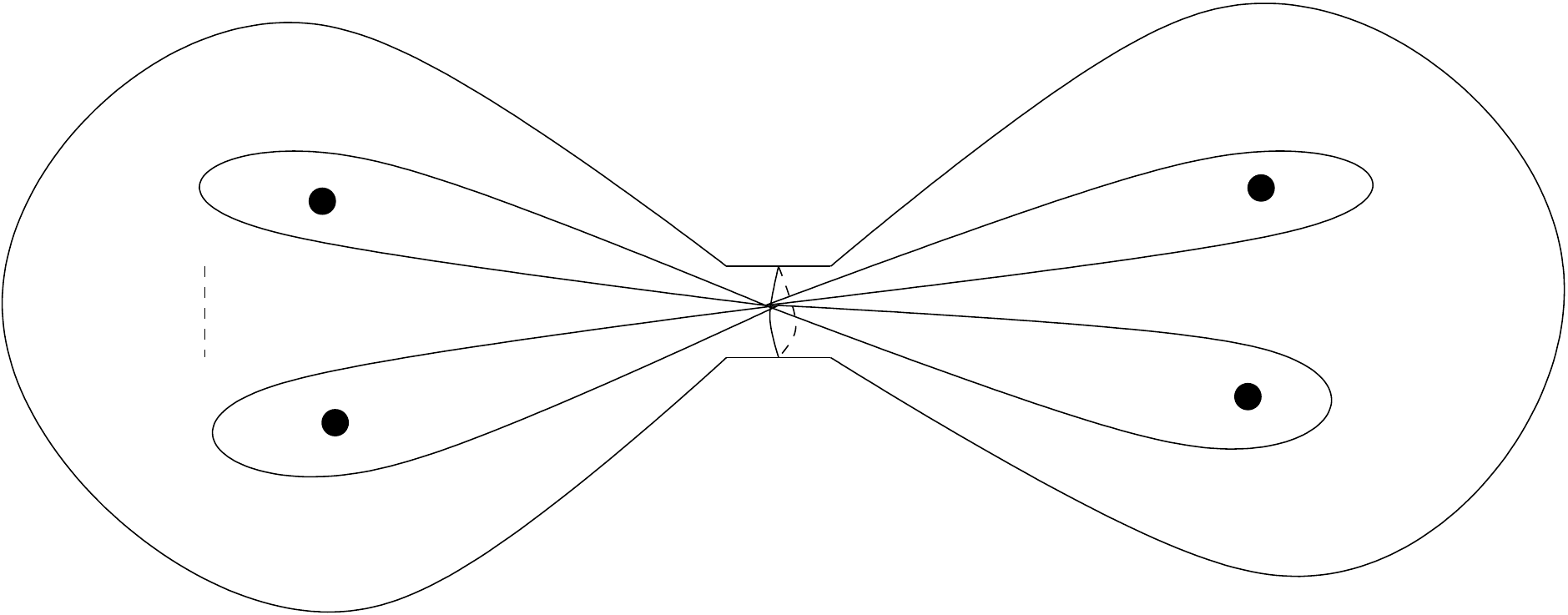}
		\put(240,260){$p_{1}$}
		\put(240,130){$p_{b-2}$}
		\put(720,260){$p_{b-1}$}
		\put(730,130){$p_{b}$}
		\put(500,250){$\beta$}
		\put(70,280){$\gamma_{1}$}
		\put(60,120){$\gamma_{b-2}$}
		\put(880,280){$\gamma_{b-1}$}
		\put(880,120){$\gamma_{b}$}
	\end{overpic}
    \caption{\label{degenerate} The target $\PP^1$ degenerates when two branch points meet}
    \end{figure}

Without loss of generality, suppose $\tau_{b} = (12)$, i.e. it switches the first two sheets of the cover, and suppose $\tau_{b-1} = (ij)$. There are three cases:

(1) If $\tau_{b-1} = \tau_b$, then $\tau = \mbox{id}$. Consequently over the node $s$, we see $k$ nodes $t_1, \ldots, t_k$ arising in the degenerate cover, each of which is unramified;

(2) If $|\{i,j\} \cap \{1,2\}| = 1$, say $(ij) = (13)$, then $\tau = (123)$, i.e. it switches the first three sheets while fixing the others. We see $k-2$ nodes $t_1, \ldots, t_{k-2}$ arising in the degenerate cover, such that $t_1$ has order of ramification $3-1 = 2$ and $t_2, \ldots, t_{k-2}$ are unramified;

(3) If $\{i,j\} \cap \{1,2\} = \emptyset$, say $(ij) = (34)$, then $\tau = (12)(34)$. We see $k-3$ nodes $t_1, \ldots, t_{k-2}$ arising in the degenerate cover, such that $t_1$ and $t_2$ are both simply ramified and $t_3, \ldots, t_{k-2}$ are unramified.

Let $f: Z\to \Mbar_g$ be the moduli map sending a branched cover to (the stable limit of) its domain curve. The intersection $f_{*}Z\cdot \delta$ can be read off from the description of admissible covers. For instance, if an admissible cover belongs to case (1), it possesses $k$ unramified nodes. Over the $\PP^1$-component containing $p_{b-1}$ and $p_b$, there are $k-2$ rational tails that map isomorphically as well as a rational bridge that admits a double cover.
Blowing down a rational tail gives rise to a smooth point of the stable limit, hence only the rational bridge contributes to the intersection with $\delta$ and the contribution is $2$, since it is a $(-2)$-curve. The intersection $f_{*}Z\cdot \lambda$ can be deduced from the relation $12\lambda = \delta + \omega^2$, where $\omega$ is the first Chern class of the relative dualizing sheaf of the universal covering curve over $Z$, cf. \cite{Moduli}*{Chapter 6.C} for a sample calculation.

Using these ideas, Harris and the third author obtained a slope formula for $s_Z$ \cite{HarrisMorrisonSlopes}*{Corollary 3.15} in terms of counts of branched covers in each of the three cases, for which they provide only recursive formulae in terms of characters of the symmetric group. More generally, enumerating non-isomorphic branched covers with any fixed ramification is a highly non-trivial combinatorial \emph{Hurwitz counting problem}.

Consider the case when $2k \geq g+2$. Since the Brill-Noether number $\rho(g, 1, k)\geq 0$, a general curve of genus $g$ admits a $k$-sheeted cover of $\PP^1$. Therefore, $Z$ is a moving curve in $\Mbar_g$. Assuming that, for $g$ large, all ordered pairs of simple transpositions are equally likely to occur as $(\tau_{b-1},\tau_{b})$---which seems plausible to first order in $k$---leads to the estimate (cf. \cite{HarrisMorrisonSlopes}*{Remark 3.23}) $s_Z \simeq \frac{576}{5g}$ (plus terms of lower order in $g$) as $g \to \infty$. 

\subsection{Linear sections of the Severi variety}
\label{Fedorchuk}

A branched cover of $\PP^1$ can be regarded as a map to a one-dimensional projective space. One way to generalize is to consider curves in $\PP^2$.
Let $\PP(d) = \PP^{{d+2\choose 2} - 1}$ be the space of plane curves of degree $d$. Consider the \emph{Severi variety} $V_{\irr}^{d, n}\subset \PP(d) $ defined as the closure of the locus parameterizing degree $d$, irreducible, plane nodal curves with $n$ nodes. The dimension of $V_{\irr}^{d, n}$ is $N = 3d + g - 1$, where $g = {d-1\choose 2}  - n$ is the geometric genus of a general curve in $V_{\irr}^{d, n}$. Let $H_p\subset \PP(d)$ be a hyperplane parameterizing curves that pass through a point $p$ in $\PP^2$. Now fix $N - 1$ general points $p_1, \ldots, p_{N-1}$ in the plane. Consider the one-dimensional section of $V_{\irr}^{d, n}$ cut out by these hyperplanes:
$$ C_{\irr}^{d, n} = V_{\irr}^{d, n} \cap H_{p_1} \cap \cdots \cap H_{p_{N-1}}. $$

Normalizing the nodal plane curves as smooth curves of genus $g$, we obtain a moduli map from $C_{\irr}^{d, n}$ to $\Mbar_g$ (after applying stable reduction to the universal curve). The calculation for the slope of $C_{\irr}^{d, n}$ was carried out by Fedorchuk \cite{FedorchukSeveri}. The intersection $C_{\irr}^{d, n}\cdot \delta$ can be expressed by the degree of Severi varieties. For instance, let $N_{\irr}^{d, n}$ be the degree of $V_{\irr}^{d, n}$ in $\PP(d)$. Then $N_{\irr}^{d, n+1}$ corresponds to the number of curves in $C_{\irr}^{d, n}$ that possesses $n+1$ nodes. Each such node contributes $1$ to the intersection with $\deltairr$. Therefore, we have
$$C_{\irr}^{d, n}\cdot \deltairr = (n+1) N_{\irr}^{d, n+1}. $$
Moreover, the degree of Severi varieties was worked out by Caporaso and Harris \cite{CaporasoHarris}*{Theorem 1.1}, though again only a recursive formula is known.

The calculation of $C_{\irr}^{d, n}\cdot \lambda$ is much more involved. Based on the idea of \cite{CaporasoHarris}, fix a line $L$ in $\PP^2$ and
consider the locus of $n$-nodal plane curves whose intersections with $L$ are of the same type, namely, intersecting $L$ transversely at $a_1$ fixed points and
at $b_1$ general points, tangent to $L$ at $a_2$ fixed points and at $b_2$ general points, etc. The closure of this locus is called the \emph{generalized Severi variety}.
A hyperplane section of the Severi variety, as a cycle, is equal to a union of certain generalized Severi varieties \cite{CaporasoHarris}*{Theorem 1.2} and 
$C_{\irr}^{d, n}$ degenerates to a union of linear sections of generalized Severi varieties. The difficulty arising from this approach is that the surface of the total degeneration admits only a rational map to $\Mbar_g$, which does not a priori extend to a morphism, as would be the case over a one-dimensional base. Treating a plane curve as the image of a stable map, Fedorchuk was able to resolve the indeterminacy of this moduli map, take the discrepancy into account, and eventually express $C_{\irr}^{d, n}\cdot \lambda$ as a recursion \cite{FedorchukSeveri}*{Theorem 1.11}.

When the Brill-Noether number $\rho(g, 2, d) = 3d - 2g - 6$ is non-negative, a general curve of genus $g$ can be realized as a plane nodal curve of degree $d$. In this case $C_{\irr}^{d, n}$ yields a moving curve in $\Mbar_g$. Fedorchuk evaluated the slope of $C_{\irr}^{d, n}$ explicitly for $d\leq 16$ and $g\leq 21$, cf. \cite{FedorchukSeveri}*{Table 1}, which consequently serves as a lower bound for $s(\Mbar_g)$. In this range, the bounds decrease from $10$ to $4.93$ and Fedorchuk speculates that ``even though we have nothing to say about the asymptotic behavior of the bounds produced by curves $C_{\irr}^{d, n}$, it would not be surprising if these bounds approached $0$, as $g$ approached $\infty$''.

\subsection{Imposing conditions on canonical curves}
\label{CoskunHarrisStarr}

We have discussed covers of $\PP^1$ and curves in $\PP^2$ as means of producing moving curves in $\Mbar_g$. What about curves in higher dimensional spaces? A natural way of embedding non-hyperelliptic curves is via their \emph{canonical model}. Coskun, Harris and Starr carried out this approach and obtained sharp lower bounds for $s(\Mbar_g)$ up to $g \leq 6$
\cite{CoskunHarrisStarrEffective}*{\S\S~2.3}.

Let us demonstrate their method for the case $g=4$. A canonical genus $4$ curve is a degree $6$ complete intersection
in $\PP^3$, cut out by a quadric surface and a cubic surface. The dimension of the family of such canonical curves is equal to $24$. Passing through a point imposes two conditions to a curve in $\PP^3$ and intersecting a line imposes one condition. Now consider the one-dimensional family $B$ parameterizing genus $4$ canonical curves that pass through $9$ general fixed points and intersect $5$ general lines. Note that  $9$ general points uniquely determine a \emph{smooth} quadric $Q$ containing them, and $5$ general lines intersect $Q$ at $10$ points. Let $C$ be a curve in the family parameterized by $B$. If $C$ is not contained in $Q$, then it has to intersect $Q$ at $\geq 9 + 5 = 14$ points, contradicting that $C\cdot Q = 12$. Therefore, every curve in $B$ is contained in $Q$. Recall that the Gieseker-Petri divisor $\overline{\mathcal{GP}}_{4, 3}^1$ on $\Mbar_4$ parameterizes genus $4$ curves whose canonical images lie in a quadric cone, and its slope is $\frac{17}{2}$. Therefore, the image of $B$ in $\Mbar_4$ and the divisor $\overline{\mathcal{GP}}_{4, 3}^1$ are disjoint. Moreover, since the points and lines are general, $B$ is a moving curve in $\Mbar_4$. As a consequence, $s_B = \frac{17}{2}$ is a lower bound for $s(\Mbar_4)$. This bound is sharp and is attained by the Gieseker-Petri divisor of genus $4$ curves lying on singular quadric surface in $\PP^3$ .

In general, the dimension of the Hilbert scheme of genus $g$ canonical curves in $\PP^{g-1}$ is $g^2 + 3g - 4$. Since
$g^2 + 3g - 4 = (g+5)(g-2) + 6$, we get a moving curve $B$, for $g \ge 9$ from the canonical curves that contain $g+5$ general points and intersect a general linear subspace
$\PP^{g-7}$. Several difficulties arise in trying to imitate the calculation of the slope of $B$. We have no detailed description of the geometry of canonical curves for large $g$, and especially of their enumerative geometry. However, this approach is sufficiently intriguing for us to propose the following:

\begin{problem}\label{canonicalslopes} Determine the lower bounds for $s(\Mbar_g)$ resulting  from computing the characteristic numbers of canonical curves of arbitrary genus $g$.
\end{problem}

\subsection{Descendant calculation of Hodge integrals}
\label{Pandharipande}

We have seen a number of explicit constructions of moving curves in $\Mbar_g$. A rather different construction was investigated by Pandharipande \cite{PandharipandeSlope}
via Hodge integrals on $\Mbar_{g,n}$. Let $\psi_i$ be the first Chern class of the cotangent line bundle on $\Mbar_{g,n}$ associated to the $i$th marked point. It is well known that
$\psi_i$ is a nef divisor class. Since a nef divisor class is a limit of ample classes, any curve class of type
$$\psi_1^{a_1}\cdots \psi_n^{a_n}, \quad \sum_{i=1}^n a_i  = 3g- 4 + n $$
is a moving curve class in $\Mbar_{g,n}$. Pushing forward to $\Mbar_g$, we obtain a moving curve in $\Mbar_g$ whose slope is equal to
\begin{equation}
\label{psiintersection}
 \frac{\int_{\Mbar_{g,n}} \psi_1^{a_1}\cdots \psi_n^{a_n} \cdot \delta}{\int_{\Mbar_{g,n}} \psi_1^{a_1}\cdots \psi_n^{a_n} \cdot \lambda}.
\end{equation}

In general, such an integral given by the intersection of tautological classes on $\Mbar_{g,n}$ is called \emph{Hodge integral}.
Pandharipande evaluated \eqref{psiintersection} explicitly for $n = 1$ and $a_1 = 3g-3$. The calculation was built on some fundamental results of Hodge integrals from  \cite{FaberPandharipandeHodge}. For example, normalize a one-nodal, one-marked irreducible curve of arithmetic genus $g$ to a smooth curve of genus $g-1$ with three marked points corresponding to the original marked point and the inverse images of the node. Then we have
$$ \int_{\Mbar_{g,1}}\psi_1^{3g-3} \cdot \deltairr = \frac{1}{2}\int_{\Mbar_{g-1,3}} \psi_1^{3g-3}, $$
where the coefficient $\frac{1}{2}$ is because the normalization map $\Mbar_{g-1, 3} \to \Deltairr\subset \Mbar_{g, 1}$ is generically two to one.
The Hodge integral on the right, as well as that in the denominator of (\ref{psiintersection})  are calculated in \cite{FaberPandharipandeHodge}.

Putting everything together, Pandharipande obtains the lower bound, for all $ g \ge 2$,
$$ s(\Mbar_g) \geq \frac{60}{g+4}. $$
It remains an interesting question to calculate \eqref{psiintersection} for general $a_1, \ldots, a_n$, but based on low genus experiments, it seems that the case $n=1$ and $a_1 = 3g-3$ provides the best lower bound (cf. \cite{PandharipandeSlope}*{\S 5, Conjecture 1}). So, any new bound arising from this approach would be most likely of size $O(\frac{1}{g})$ as $g$ tends to $\infty$.

\subsection{Covers of elliptic curves and Teichm\"uller curves}
\label{Chen}

Recall that \cite{HarrisMorrisonSlopes} constructs moving curves using covers of $\PP^1$. What about covers of curves of higher genera? Suppose the domain curve has genus $g$ and the target has genus $h$. If $h > 1$, by the Riemann-Hurwitz formula, a $d$-sheeted cover satisfies $2g-2 \geq d (2h-2)$, hence there are only finitely many choices for $d$. An easy dimension count shows that the Hurwitz space parameterizing all such covers (under the moduli map) is a union of proper subvarieties of $\Mbar_g$. In principle there could exist an effective divisor containing all of those subvarieties.

This leaves the case $h = 1$, which was studied by the first author in
\cite{ChenCovers}. Let $\mu = (m_1, \ldots, m_k)$ be a partition of $2g-2$. Consider the Hurwitz space $\TT_{d,\mu}$ parameterizing degree $d$, genus $g$, connected covers $\pi$ of elliptic curves with a \emph{unique} branch point $b$ at the origin whose ramification profile is given by $\mu$, i.e.
$$ \pi^{-1}(b) = (m_1 + 1) p_1 + \cdots + (m_k + 1) p_k + q_1 + \cdots + q_l, $$
where $p_i$ has order of ramification $m_i$ and $q_j$ is unramified.
Over a fixed elliptic curve $E$, there exist finitely many non-isomorphic such covers. If we vary the $j$-invariant of $E$, the covering curves also vary to form
a one-dimensional family, namely, the Hurwitz space $\TT_{d,\mu}$ is a curve.

The images of $\TT_{d,\mu}$ for all $d$ form a countable union of curves in $\Mbar_g$. We will see below that $\TT_{d,\mu}$ is a \emph{Teichm\"uller curve} hence \emph{rigid}. Nevertheless if $k \geq g-1$, the union $\cup_d \TT_{d, \mu}$ forms a
Zariski dense subset in $\Mbar_g$ \cite{ChenCovers}*{Proposition 4.1}. In this case the $\liminf$ of the slopes of $\TT_{d, \mu}$ as $d$ approaches $\infty$ still provides a lower bound for $s(\Mbar_g)$ and since effective divisor can contain only finitely many $\TT_{d, \mu}$.

To calculate the slope of $\TT_{d, \mu}$, note that an elliptic curve can degenerate to a one-nodal rational curve, by shrinking
a vanishing cycle $\beta$ to the node. The monodromy action associated to $\beta$ determines the topological type of the admissible cover arising in the degeneration, which in principle indicates how to count the intersection number $\TT_{d, \mu}\cdot \delta$. Moreover, using the relation $12\lambda = \delta + \omega^2$ associated to the universal covering map, one can also calculate $\TT_{d, \mu}\cdot \lambda$. This leads to a formula, again recursive and difficult to unwind for the same reasons as the formulae in \cite{HarrisMorrisonSlopes}, for the slope of  $\TT_{d, \mu}$ in \cite{ChenCovers}*{Theorem 1.15}.

However, $\TT_{d, \mu}$ can be regarded as a special Teichm\"uller curve, and this provides a whole new perspective. Let $\HH(\mu)$ parameterize pairs $(X, \omega)$ such that $X$ is a genus $g$ Riemann surface, $\omega$ is a holomorphic one-form on $X$ and $(\omega)_0 = \sum m_i p_i$ for distinct points $p_i \in X$. Note that integrating $\omega$ along a path connecting two points defines a \emph{flat} structure on $X$. In addition, integrating $\omega$ along a basis of the relative homology group $H_1(X; p_1, \ldots, p_k)$ realizes $X$ as a plane polygon with edges identified appropriately under affine translation. The reader may refer to \cite{ZorichFlat} for an excellent introduction to flat surfaces. Varying the shape of the polygon induces an SL$_2(\RR)$ action on $\HH(\mu)$. Project an orbit of this action to $\M_g$ by sending $(X, \omega)$ to $X$. The image is called a \emph{Teichm\"uller curve} if it is algebraic. Teichm\"uller curves possess a number of fascinating properties. They are geodesics under the Kobayashi metric on $\M_g$. They are rigid \cite{McMullenRigid}, hence give infinitely many examples of rigid curves on various moduli spaces, in particular, on the moduli space of pointed rational curves \cite{ChenRigid}.

Consider a branched cover $\pi: X\to E$ parameterized in $\TT_{d, \mu}$, where $E$ is the square torus with the standard one-form $dz$. The pullback $\omega = \pi^{-1}(dz)$
has divisor of zeros $\sum_i m_i p_i$. Therefore, $(X, \omega)$ yields a point in $\HH(\mu)$. Alternatively, one can glue $d$ copies of the unit square to realize $X$ as a \emph{square-tiled surface}  endowed with flat structure;  Figure~\ref{square-tiled} shows an example.

\begin{figure}[ht]
    \centering
    \begin{overpic}[scale=0.7]{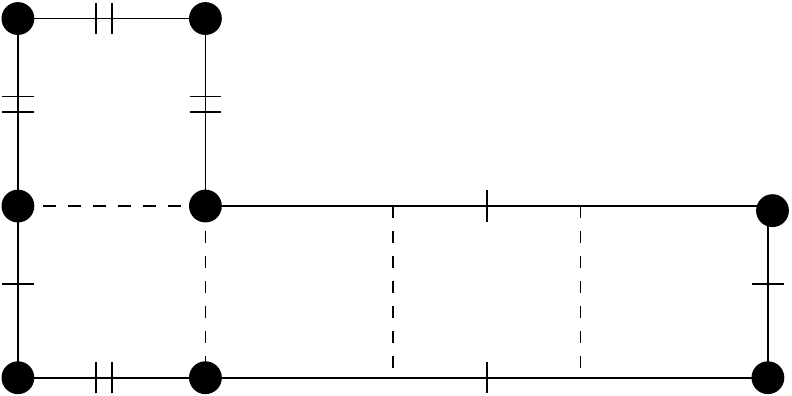}
		\put(130,350){$5$}
		\put(130,120){$1$}
		\put(360,120){$2$}
		\put(600,120){$3$}
		\put(830,120){$4$}
	\end{overpic}
    \caption{\label{square-tiled} A square-tiled surface for $g=2$, $d=5$ and $\mu = (2)$}
   \end{figure}

The SL$_2(\RR)$ action amounts to deforming the square to a rectangle, i.e. changing the $j$-invariant of $E$, hence the Hurwitz curve $\TT_{d, \mu}$ is an invariant orbit under this action. Indeed, $\TT_{d, \mu}$ is called an \emph{arithmetic Teichm\"uller curve}. Note that there exist Teichm\"uller curves that do not arise from a branched cover construction and their classification is far from complete. We refer to \cite{MoellerPCMI} for a survey on Teichm\"uller curves from the viewpoint of algebraic geometry. 

As a square-tiled surface, $X$ decomposes into horizontal cylinders with various heights and widths, which are bounded by horizontal geodesics connecting two zeros of
$\omega$. For instance, the top square of the surface in Figure~\ref{square-tiled} admits a horizontal cylinder of height and width both equal to $1$, while the bottom four squares form another horizontal cylinder of height $1$ and width $4$. Suppose the vanishing cycle of $E$ is represented by the horizontal core curve of the square. Then the core curve of a horizontal cylinder of height $h$ and width $l$ shrinks to $h$ nodes, each of which contributes $\frac{1}{l}$ to the intersection $\TT_{d, \mu} \cdot \delta$ by a local analysis. Denote by $\frac{h}{l}$ the \emph{modulus} of such a cylinder. In general, the ``average'' number of modulus of horizontal cylinders in an SL$_2(\RR)$-orbit closure is defined as the \emph{Siegel-Veech area constant} $c_{\mu}$ associated to the orbit. See \cite{EskinMasurZorichSV} for a comprehensive introduction to Siegel-Veech constants.

The slope of $\TT_{d, \mu}$ has an expression involving its Siegel-Veech constant \cite{ChenRigid}*{Theorem 1.8}, which also holds for any Teichm\"uller curve \cite{ChenMoellerSumlyap}*{\S 3.4}. Based on massive computer experiments, Eskin and Zorich believe that the Siegel-Veech constant $c_{\mu}$ approaches $2$ as $g$ tends to $\infty$ for Teichm\"uller curves in any (non-hyperelliptic) $\HH(\mu)$. Assuming this expectation, the slope formula cited above implies that $s(\TT_{d, \mu})$ grows as $\sim \frac{576}{5g}$ for $g\gg 0$ and $k\geq g-1$. We have seen this bound $\frac{576}{5g}$ in Section~\ref{HarrisMorrison}. But the curves used in \cite{HarrisMorrisonSlopes} are moving while Teichm\"uller curves are rigid! It would be interesting to see whether this is only a coincidence having to do with some property of branched covers or whether the bound $\frac{576}{5g}$ has a more fundamental ``hidden'' meaning.

A final remark on Teichm\"uller curves is that their intersection numbers with divisors on $\Mbar_g$ can provide information about
the SL$_2(\RR)$ dynamics on $\HH(\mu)$. For instance, about a decade ago Kontsevich and Zorich conjectured, based on numerical data, that for many low genus strata $\HH(\mu)$ the sums of Lyapunov exponents are the same for  
\emph{all} Teichm\"uller curves contained in that stratum. This conjecture has been settled by the first author and M\"oller \cite{ChenMoellerSumlyap}. The idea is that the three quantities---the slope, the Siegel-Veech constant and the sum of Lyapunov exponents---determine each other, hence it suffices to show the slopes are non-varying for all Teichm\"uller curves in a low genus stratum. 

Consider, as an example, $\HH(3,1)$ in genus $g=3$. If a curve $C$ possesses a holomorphic one-form $\omega$ such that $(\omega)_0 = 3p_1 + p_2$ for $p_1\neq p_2$, then $C$ is not hyperelliptic, since the hyperelliptic involution would switch the zeros. Consequently 
a Teichm\"uller curve in $\HH(3,1)$ is disjoint from the divisor $\M_{3,2}^1$ of hyperelliptic curves in $\M_3$. Checking that this remains true for the respective closures of these two loci immediately implies that 
the slopes of all Teichm\"uller curves in $\HH(3,1)$ are equal to $s(\Mbar_{3,2}^1) = 9$. For a detailed discussion of the interplay between Teichm\"uller curves and the Brill-Noether divisors, see \cite{ChenMoellerSumlyap}. Using similar ideas, the first author and M\"oller also settled the case of Teichm\"uller curves generated by quadratic differentials in low genus in \cite{ChenMoellerLyapquad}.

\subsection{Moduli spaces of $k$-gonal curves}
\label{analogy}

We end this section by discussing various questions related to slopes on moduli spaces of $k$-gonal curves and we begin with the case of hyperelliptic curves.

Let $\overline{H}_g:=\Mbar_{g, 2}^1$ be the closure of locus of genus $g$ hyperelliptic curves in $\Mbar_g$. Alternatively, it is the admissible cover compactification of the space of genus $g$, simply branched double covers of $\PP^1$. We have an injection $ \iota: \overline{H}_g \hookrightarrow \Mbar_g. $
The rational Picard group of $\overline{H}_g$ is generated by boundary components $\Xi_0,\ldots, \Xi_{[\frac{g-1}{2}]}$ and $\Theta_1, \ldots, \Theta_{[\frac{g}{2}]}$ (cf. \cite{Moduli}*{Chapter 6.C}) and \cite{KeelMcKernanContractible} shows that these classes also generate $\eff{\overline{H}_g}$. A general point of $\Xi_i$ parameterizes a double cover of a one-nodal union $\PP^1 \cup \PP^1$ branched at $2i+2$ points in one component and $2g-2i$ in the other. A general point of $\Theta_i$ parameterizes a double cover of $\PP^1\cup \PP^1$ branched at $2i+1$ points in one component and $2g-2i+1$ in the other. Cornalba and Harris \cite{CornalbaHarrisAmple} proved the following formulae:
$$\iota^{*} (\Deltairr) = 2\sum_{i=0}^{\lfloor \frac{g-1}{2}\rfloor} \Xi_{i}, \quad \iota^{*}(\Delta_i) = \frac{1}{2}\Theta_{2i+1} \ \mbox{for } i\geq 1, $$
$$ \iota^{*}(\lambda) = \sum_{i=0}^{\lfloor\frac{g-1}{2}\rfloor}\frac{(i+1)(g-i)}{4g+2} \Xi_{i} + \sum_{i = 1}^{\lfloor \frac{g}{2}\rfloor} \frac{i(g-i)}{4g+2}\Theta_{i}. $$
Since the smallest boundary coefficient in the expression of $\iota^{*}(\lambda)$ is $\frac{g}{(4g+2)}$ for $\Xi_0 \equiv \iota^{*}\bigl(\frac{\Deltairr}{2}\bigr)$ (modulo higher boundary terms),
if we impose formally the same slope problem to $\overline{H}_g$, the lower bound for slopes of effective divisors on $\overline{H}_g$ is
$$ 2 \cdot \frac{4g+2}{g} = 8 + \frac{4}{g}. $$
Phrasing it differently, if $B\subset \overline{H}_g$ is any one-dimensional family of genus $g$ curves whose general member is smooth and hyperelliptic, then $(8g+4) B\cdot \lambda\geq g B\cdot \delta$, therefore  $s_B\leq 8 + \frac{4}{g}$, cf. \cite{CornalbaHarrisAmple} or \cite{Moduli}*{Corollary 6.24}. Note that this bound converges to $8$ as $g$ approaches $\infty$. The maximum
$8+\frac{4}{g}$ can be achieved by considering a Lefschetz pencil of type $(2, g+1)$ on the quadric $\mathbb P^1\times \mathbb P^1$.
Similar bounds were obtained for trigonal families by Stankova \cite{Stankova}. If $B\subset \Mbar_{g, 3}^1$ is any one-dimensional family of trigonal curve with smooth generic member, then 
$$s_B\leq \frac{36(g+1)}{5g+1}.$$ A better bound $s_B\leq 7+\frac{6}{g}$ is known to hold for trigonal families $B\subset \Mbar_g$ not lying in the \emph{Maroni locus} of $\Mbar_{g, 3}^1$, that is, the subvariety of the trigonal locus corresponding to curves with unbalanced scroll invariants. It is of course highly interesting to find such bounds for the higher $k$-gonal strata
$\Mbar_{g, k}^1$. A yet unproven conjecture of Harris, see \cite{Stankova}*{Conjecture 13.3}, predicts that if $B\subset \Mbar_{g, k}^1$ is any $1$-dimensional family with smooth generic member and not lying in a
codimension one subvariety of $\Mbar_{g, k}^1$, then:
$$s_B\leq \Bigl(6+\frac{2}{k-1}\Bigr)+\frac{2k}{g}.$$

To this circle of ideas belongs the following fundamental question:
\begin{problem} Fix $g$ sufficiently large, so that $\Mbar_g$ is of general type. Find the smallest integer $k_{g, \mathrm{max}}\leq \frac{g+2}{2}$ such that $\Mbar_{g, k}^1$ is a variety of general type for a $k\geq k_{g, \mathrm{max}}$. Similarly, find the largest integer $2\leq k_{g, \mathrm{min}}$ such that $\Mbar_{g, k}^1$ is uniruled for all $k\leq k_{g, \mathrm{min}}$. 
Is it true that 
$$\liminf_{g\rightarrow \infty} \ k_{g, \mathrm{min}}=\liminf_{g\rightarrow \infty} \ k_{g, \mathrm{max}}?$$
Obviously a similar question can be asked for the Severi varieties $\Mbar_{g, d}^2$, or indeed for all Brill-Noether subvarieties of the moduli space.
To highlight our ignorance in this matter, Arbarello and Cornalba \cite{ArbarelloCornalbaFootnotes}, using a beautiful 
construction of Beniamino Segre, showed that $\Mbar_{g, k}^1$ is unirational for $k\leq 5$, but it is not even known that $k_{g, \mathrm{min}}\geq 6$ for arbitrarily large $g$. The best results in this direction are due to Geiss \cite{Geiss} who proves the unirationality of $\Mbar_{g, 6}^1$ for most genera $g\leq 45$.
\end{problem}

\section{The slope of $\A_g$}

In view of the tight analogy with the case of $\Mbar_g$, we want to discuss questions related to the slope of the moduli space $\cA_g$ of principally polarized abelian varieties (ppav) of dimension $g$. Let $\overline{\cA}_g$ be the perfect
cone, or first Voronoi, compactification of $\A_g$ and denote by  $D:=\overline{\cA}_g-\cA_g$ the irreducible boundary divisor. Then $\mbox{Pic}(\overline{\cA}_g)=\mathbb Q\cdot \lambda_1\oplus \mathbb Q\cdot [D]$, where $\lambda_1:=c_1(\mathbb E)$ is the first Chern class of the Hodge bundle. Sections of $\mbox{det}(\mathbb E)$ are weight $1$ Siegel modular forms. Shepherd-Barron \cite{ShepherdBarron} showed that for $g\geq 12$, the perfect cone compactification is the canonical model of $\cA_g$. 

In analogy with the case of $\Mbar_g$, we define the slope of an effective divisor $E\in \mathrm{Eff}(\aa_g)$ as
$$s(E):=\inf\left\{\frac{a}{b}: a, b>0, \ a\lambda_1 -b[D]-[E]= c[D], \ c>0\right\},$$
and then the slope of the moduli space as the quantity
$$s(\aa_g):=\inf_{E\in \mathrm{Eff}(\aa_g)} s(E).$$ Since $K_{\aa_g}=(g+1)\lambda_1-[D]$, it follows that $\aa_g$ is of general type if $s(\aa_g)<g+1$, and $\aa_g$ is uniruled when $s(\aa_g)>g+1$. Mumford \cite{MumfordRavello} was the first to carry out divisor class calculations in $\aa_g$. In particular, he studied the \emph{Andreotti-Mayer divisor} $N_0$ on $\cA_g$ consisting of ppav
$[A, \Theta]$ having a singular theta divisor. Depending on whether the singularity occurs at a torsion point or not, one distinguishes between the components 
$\theta_{\mathrm{null}}$ and $N_0'$ of the Andreotti-Mayer divisor. The following scheme-theoretical equality holds:
$$N_0=\theta_{\mathrm{null}}+2N_0'.$$
The cohomology classes of the components of $\overline{N}_0$ are given are computed in \cite{MumfordRavello}: 
$$[\overline{N}_0']=\Bigl(\frac{(g+1)!}{4}+\frac{g!}{2}-2^{g-3}(2^g+1)\Bigr)\lambda_1-\Bigl(\frac{(g+1)!}{24}-2^{2g-6}\Bigr)[D],$$
respectively 
$$[\overline{\theta}_{\mathrm{null}}]=2^{g-2}(2^g+1)\lambda_1-2^{2g-5}[D].$$
Using these formulas coupled with Tai's results on the singularities of $\aa_g$, Mumford concluded that $\aa_g$ is of general type for $g\geq 7$. Note that at the time, the result had already been established for $g\geq 9$ by Tai \cite{TaiKodaira} and by Freitag \cite{Freitag} for $g=8$. On the other hand, it is well-known that $\cA_g$ is unirational for $g\leq 5$. The remaining case is notoriously difficult. This time, the three authors refrain from betting on possible outcomes and pose the: 
\begin{problem} What is the Kodaira dimension of $\cA_6$? 
\end{problem}

The notion of slope for $\aa_g$ is closely related to that of $\Mbar_g$ via the \emph{Torelli map}
$$ \tau: \Mbar_g \dashrightarrow \aa_g, $$
sending a curve to its (generalized) Jacobian. The restriction of $\tau$ to the union of $\M_g$ and the locus of one-nodal irreducible curves is an embedding.
If a curve $C$ consists of a one-nodal union of two lower genus curves $C_1$ and $C_2$, then $J(C) \cong J(C_1)\times J(C_2)$, which does not depend on the position of the node. Therefore, $\tau$ contracts $\Delta_i$ for $i > 0$ to a higher codimension locus in $\aa_g$. Moreover, we have that
$$ \tau^{*}(\lambda_1) = \lambda, \quad \tau^{*}(D) = \deltairr. $$
Therefore, if $F$ is an effective divisor on $\aa_g$ and that does not contain $\tau(\M_g)$, then $\tau^{*}(F)$ and $F$ have the same slope. If we know that $s(\Mbar_g)\geq \epsilon$ for a positive number $\epsilon$, then any modular form of weight smaller than $\epsilon$ has to vanish on $\tau(\M_g)$. This would provide a novel approach  to understand which modular forms cut out $\M_g$ in $\A_g$, and thus give a solution to the geometric \emph{Schottky problem}. By work of Tai \cite{TaiKodaira} (explained in \cite{GrushevskyAbelian}*{Theorem 5.19}), the lower bound for slopes of effective divisors on $\aa_g$ approaches $0$ as $g$ tends to
$\infty$, that is, $\mbox{lim}_{g\rightarrow \infty} s(\aa_g)=0$. In fact, there exists an effective divisor
on $\aa_g$ whose slope is at most
$$\sigma_g := \frac{(2\pi)^2}{\big(2(g!)\zeta(2g)\big)^{1/g}}.$$
Since $\zeta(2g) \to 1$ and $(g!)^{1/g} \to \frac{g}{e}$ as $g \to \infty$, we find that $g\, \sigma_g \to 68.31\ldots$. Although this is a bit bigger than the unconditional lower bound of $60$ for $g \,s(\Mbar_g)$ of subsection \ref{Pandharipande}, this shows that 
the slope of this divisor is smaller than the heuristic lower bound $\frac{576}{5g}$ for $s(\Mbar_g)$ emerging from \cite{HarrisMorrisonSlopes} and \cite{ChenRigid} for large values of $g$.

As for $\Mbar_g$, the slope of $\aa_g$ is of great interest. The first case is $g=4$ and it is known \cite{smmodfour} that $s(\aa_4)=8$, and the minimal slope is computed by the divisor $\overline{\tau(\M_4)}$ of genus $4$ Jacobians. In the next case $g=5$, the class of the closure of the Andreotti-Mayer divisor is 
$$[\overline{N}_0']=108\lambda_1 -14D,$$ giving  the upper bound $s(\aa_5)\leq \frac{54}{7}$.
Very recently, a complete solution to the slope question on $\aa_5$ has been found in \cite{FGSMV} by Grushevsky, Salvati-Manni, Verra and the second author. We spend the rest of this section explaining the following result:
\begin{theorem}\label{slopea5}
The slope of $\aa_5$ is attained by $\overline{N_0'}$. That is, $s(\aa_5)=\frac{54}{7}$. Furthermore, $\kappa(\aa_5, \overline{N}_0')=0$, that is,
the only effective divisors on $\aa_5$ having minimal slope are the multiples of $\overline{N}_0'$.
\end{theorem} 

The proof relies  on the intricate geometry of the generically $27:1$ Prym map 
$$P:\overline{\mathcal{R}}_6\dashrightarrow {}\aa_5{}.$$ The map $P$ has been investigated in detail in \cite{DonagiSmith} and it displays some breathtakingly
beautiful geometry. For instance the Galois group of $P$ is the Weyl group of $E_6$, that is, the subgroup of $\mathfrak S_{27}$ consisting of permutations preserving the
intersection product on a fixed cubic surface. The divisor $N_0'$ is the branch locus of $P$, whereas the ramification divisor $\mathcal{Q}$ has three alternative realizations as a geometric subvariety of $\mathcal{R}_6$, see \cite{FarkasVerraNik} and \cite{FGSMV} for details. One should view this statement as a Prym analogue of the various incarnations of the $K3$ divisor
$\mathcal{K}_{10}$ on $\M_{10}$, see \cite{FarkasPopa}:
\begin{theorem}\label{ram}
The ramification divisor $\mathcal{Q}$ of the Prym map $P:\mathcal{R}_6\rightarrow \cA_5$ has the following geometric incarnations:
\begin{enumerate}
\item $\bigl\{[C, \eta]\in \mathcal{R}_6:\mathrm{Sym}^2 H^0(C, K_C\otimes \eta)\stackrel{\ncong}\longrightarrow H^0(C, K_C^{\otimes 2})\bigr\}.$
\item $\bigl\{[C, \eta]\in \mathcal{R}_6: C \mbox{ has a sextic plane model with a totally tangent conic}\bigr\}.$
\item $\bigl\{[C, \eta]\in \mathcal{R}_6: C \mbox{ is a section of a Nikulin surface}\bigr\}.$
\item $\bigl\{[C, \eta]\in \mathcal{R}_6: \mathrm{Sing}^{\mathrm{st}}(\Xi)\neq 0\bigr\}.$
\end{enumerate}
\end{theorem}
The first realization is the most straightforward and it relies on the description of the differential of the Prym map via Kodaira-Spencer theory, see \cite{DonagiSmith}. Description $(4)$ refers to stable singularities of the theta divisor $\Xi\subset P(C, \eta)$ associated to the Prym variety. In particular, $\Xi$ has a \emph{stable singularity} if and only if the \'etale double cover $f:\tilde{C}\rightarrow C$ induced by the half-period $\eta\in \mbox{Pic}^0(C)$, $\eta^{\otimes 2}=\mathcal{O}_C$, carries a line bundle $L$ with $\mbox{Nm}_f(L)=K_C$ with $h^0(\tilde{C}, L)\geq 4$.  Description $(3)$ concerns moduli spaces of $K3$ surfaces endowed with a symplectic involution (Nikulin surfaces): see \cite{FarkasVerraNik}. The equivalence $(1)\Leftrightarrow (4)$ can be regarded as stating that $\mathcal{Q}$ is simultaneously the Koszul  and the Brill-Noether divisors (in Prym sense) on the moduli space $\mathcal{R}_6$!
By a local analysis, if $\pi:\overline{\mathcal{R}}_6\rightarrow \Mbar_6$ is the morphism forgetting the half-period, one proves the following relation in $\mbox{Pic}(\overline{\mathcal{R}}_6)$:
\begin{equation}\label{pullback}
P^*(\overline{N}_0')= 2 \overline{\mathcal{Q}}+\overline{\mathcal{U}}+20\deltairr^{''},
\end{equation}
where $\mathcal{U}=\pi^*(\mathcal{GP}_{6, 4}^1)$ is the \emph{anti-ramification} divisor of $P$  and finally, $\deltairr^{''}$ denotes the boundary divisor class corresponding to Wirtinger
coverings. Using the different parametrizations of $\mathcal{Q}$ provided by Theorem \ref{ram}, one can construct a sweeping rational curve $R\subset \overline{\mathcal{Q}}$ such that 
$$R\cdot \overline{\mathcal{U}}=0, \  R\cdot \deltairr^{''}=0 \ \mbox{ and } R\cdot \overline{\mathcal{Q}}<0.$$
Via a simple argument, this shows that in formula (\ref{pullback}), the divisor $\overline{\mathcal{Q}}$ does not contribute to the linear system 
$|P^*(\overline{N}_0')|$. Similar arguments show that $\overline{\mathcal{U}}$ and $\deltairr^{''}$ do not contribute either, that is, $N_0'$ is the only effective divisor in its linear 
system, or equivalently $\kappa(\aa_5, \overline{N}_0')=0$. In particular $s(\aa_5)=s(\overline{N}_0')=\frac{54}{7}$. This argument shows that $\overline{N_0'}$ is rigid, hence $s'(\aa_5)>s(\aa_5)$, so we ask:

\begin{problem} What is $s'(\aa_5)$? 
\end{problem}

 A space related to both $\M_g$ and $\A_g$ is the universal theta divisors $\mathfrak{Th}_g\rightarrow \M_g$, which can be viewed as the universal degree $g-1$ symmetric product 
$\Mbar_{g, g-1}\dblq\mathfrak{S}_{g-1}$ over $\M_g$. The following result has been recently established by Verra and the second author in \cite{FarkasVerraThet}:
\begin{theorem} $\mathfrak{Th}_g$ is a uniruled variety for genus $g\leq 11$ and of general type for $g\geq 12$.
\end{theorem}
The proof gives also a description of the relative effective cone of $\mathfrak{Th}_g$ over $\Mbar_g$ as being generated by the boundary divisor $\tilde{\Delta}_{0, 2}$ corresponding to non-reduced effective divisors of degree $g-1$ and by the universal ramification divisor of the Gauss map, that is, the closure of the locus of points
$[C, x_1+\cdots+x_{g-1}]$ for which the support of the $0$-dimensional linear series $|K_C(-x_1-\cdots-x_{g-1})|$ is non-reduced. The paper \cite{FarkasVerraThet} also gives a complete birational classification of the universal symmetric product $\Mbar_{g, g-2}\dblq\mathfrak S_{g-2}$, showing that, once again, the birational character of the moduli space changes around genus $g=12$.

We close this section with a general comment. Via the Torelli map, the moduli space of curves sits between  $\overline{H}_g$ and $\aa_g$. In terms of lower bounds for slopes, $\overline{H}_g$ and $\aa_g$ behave totally different for large $g$: the former has lower bound converging to $8$, whereas the latter approaches $0$. The failure to find of effective divisors with small slope seems to suggest that $\Mbar_g$ is ``closer'' to $\overline{H}_g$, while the failure to find moving curves seems to suggest that $\Mbar_g$ is ``closer'' to $\aa_g$.

\section{Effective classes on spaces of stable curves and maps of genus $0$}

\subsection{Symmetric quotients in genus $0$} \label{symmetric}
On $\Mbar_{0, n}$ itself, the components  of the boundary $\Delta$ are the loci $\Delta_I$, indexed by $I \subset \{ 1, 2, \ldots , n\}$ (subject to the identification $\Delta_I = \Delta_{I^\vee}$ and to the stability condition that both $|I|$ and $|I^\vee|$ be at least $2$), whose general point parameterizes a reducible curve having two components meeting in a single node and the marked points indexed by $I$ on one side and those indexed by $I^\vee$ on the other. These generate, but not freely, $\pic{\Mbar_{0, n}}$, see~\cite{KeelIntersection}.

The starting point for the study of effective cones in genus $0$ is the paper of Keel and McKernan~\cite{KeelMcKernanContractible}. In it, they consider the space $\widetilde{\M}_{0, n}$ that is the quotient of $\Mbar_{0, n}$ by the natural action of $\Sn$ by permutations of the marked points. For instance $\widetilde{\M}_{0, 2g+2}$ is isomorphic to the compactified moduli space of hyperelliptic curves of genus $g$ already discussed in this survey. The boundary $\Dsym$ of $\widetilde{\M}_{0, n}$ has components $\Dsym_i$ that are simply the images of the loci $\Delta_i$ on $\Mbar_{0, n}$ defined as the union of all $\Delta_I$ with $|I|=i$ for $i$ between from $2$ and~$\lfloor\frac{g}{2}\rfloor$.

\begin{lemma} \label{effmzeroninv} Every $\Sn$-invariant, effective divisor class $D$ on $\Mbar_{0, n}$ is an effective sum of the boundary divisors $\Delta_i$
\end{lemma}

\begin{corollary} \label{effmzeronsym} The cone $\eff{\widetilde{\M}_{0, n}}$ is simplicial, and is generated by the boundary classes $\Dsym_i$.
\end{corollary}

\begin{proof}[Proof of Lemma~\ref{effmzeroninv}]  Any $\Sn$-invariant divisor  $D$ is clearly a linear combination $\sum b_i\Delta_i$ so the point is to show that, if $D$ is effective, then we can take the $b_i$ all non-negative, and this is shown by a pretty induction using test curves. We may assume that $D$ contains no $\Delta_i$ since proving the result for the $D'$ that results from subtracting all such contained components will imply the result for $D$.

As a base for the induction, pick an $n$-pointed curve $\bigl(C, [p_i, \ldots, p_n]\bigr)$ not in the support of $D$ and form a test family with base $B\cong C$ by varying $p_n$, while fixing the other $p_i$. Since $C$ is not in $D$, the curve $B$ must meet $D$ non-negatively. On the other hand, $ B\cdot \Delta_2 = (n-1)$---there is one intersection each time $p_n$ crosses one of the other $p_i$---and is disjoint from the other $\Delta_i$. Hence, $b_2 \ge 0$.

Now assume inductively that $b_{i} \ge 0$. Choose a generic curve
$$C = \bigl(C', [p_1', \ldots, p_i']\bigr) \cup \bigl(C', [p_1'', \ldots, p_{n-i}'']\bigr) $$
in $\Delta_i$ in which $q'$ on $C'$ has been glued to $q''$ on $C''$ and form the family $B \cong C''$ by keeping $q'$ and the marked points on both sides fixed but varying $q''$ (as in  \cite{Moduli}*{Example (3.136)}). As above $B \cdot D \ge 0$, $B \cdot \Delta_j= 0$ unless $j$ is either $i$ or $i+1$. And, as above, $B \cdot \Delta_{i+1} = n-i$ (we get one intersection each time $q''$ crosses a $p_k''$), but now $B$ lies \emph{in} $\Delta_i$ so to compute $B \cdot \Delta_i$ we use the standard approach of \cite{Moduli}*{Lemma (3.94)}. On the ``left'' side, the family over $B$ is $C' \times C'$ and the section corresponding to $q'$ has self-intersection $0$. On the ``right'' side, the family is $C'' \times C'' \cong \PP^1 \times \PP^1$ blown up at the points where the constant sections corresponding to the $p_k''$ meet the diagonal section corresponding to $q''$ and hence the proper transform of that section has self-intersection $\bigl(2-(n-i)\bigr)$. The upshot is that $B \cdot D = (n-i)b_{i+1}- (n-i-2)b_i$ completing the induction.
\end{proof}

In fact, this proof shows quite a bit more. It immediately gives the first inequalities in Corollary~\ref{keelmckernan4.8} and the others follow by continuing the induction and using the identifications $\Delta_i = \Delta_{n-i}$.

\begin{corollary}[\cite{KeelMcKernanContractible}*{Lemma 4.8}] \label{keelmckernan4.8} If $D =\sum b_i\Delta_i$ is an effective divisor class on $\widetilde{\M}_{0, n}$ whose support does not contain any $\Delta_i$ (or, if $D$ is nef), then $ (n-i)b_{i+1} \ge (n-i-2)b_i$ for $ 2\le i \le \lfloor\frac{n}{2}\rfloor -1$ and $ ib_{i-1} \ge (i-2)b_i$ for $ 3\le i \le \lfloor\frac{n}{2}\rfloor$.
\end{corollary}

At this point, it is natural to hope that we might be able to replace the twiddles in Corollary~\ref{effmzeronsym} with bars with a bit more work. We will see in \S\S\ref{mzeronbar} that this is far from the case.

Next, we give another application of Lemma~\ref{effmzeroninv}, also due to Keel and kindly communicated to us by Jason Starr, this time to the Kontsevich moduli spaces of stable maps $\Mbar_{0, 0}(\PP^d, d)$. A general map $f$ in $\Mbar_{0, 0}(\PP^d, d)$  has a smooth source curve $C$ with linearly non-degenerate image $f(C) \subset \PP^d$ of degree $d$ and hence is nothing more than a rational normal curve. The space $\Mbar_{0, 0}(\PP^d,1)$ is just the Grassmannian of lines in $\PP^d$. The philosophy is to view $\Mbar_{0, 0}(\PP^d, d)$ as a natural compactification of the family of such curves. For example, $\Mbar_{0, 0}(\PP^2,2)$ has an open stratum consisting of plane conics. One boundary divisor arises when the curve $C$ becomes reducible, the map $f$ has degree $1$ on each component and image consists of a pair of transverse lines. But there is a second component, in which the map $f$ degenerates to a double cover of a line in which the image is ``virtually marked'' with the two branch points. These intersect in a locus of maps from a pair of lines to a single line in which only the image of the point of intersection is ``virtually marked''.

This generalizes: $\pic{\Mbar_{0,0}(\PP^d, d)}$ is freely generated by effective classes $\Gamma_i$, the closure of the locus whose generic map has a domain with two components on which it has degrees $i$ and $d-i$ with $1 \le i \le \lfloor \frac{d}{2}\rfloor$, and a class $G$, the degenerate locus where $f(C)$ lies in a proper subspace of $\PP^d$---see \cite{PandharipandeIntersections}*{Theorem 1}.

\begin{lemma} \label{effmoobard} A class $D = aG + \sum_i b_i\Gamma_i$ is effective  if and only if $a\ge 0$ and each $b_i \ge 0$.
\end{lemma}

\begin{proof} All we need to show is that effective classes have positive coefficients. We start with $a$. Choose a general map $g:\PP^2\to \PP^d$ for which $g^{\shavedast}\bigl(\ooo(1)\bigr) = \ooo(d)$ (i.e. a generic $d+1$-dimensional vector space $V$ of degree $d$ polynomials in the plane). Then $g$ sends a general pencil $B$ of lines in $\PP^2$ to a pencil of rational curves of degree $d$. The image of a general element of this pencil will be a rational normal curve of degree $d$, hence non-degenerate, so $g(B) \not\subset G$ and hence $g(B) \cdot G \ge 0$.  No element of the pencil will be reducible, hence $g(B) \cdot \Gamma_i = 0$. Since we can make any rational normal curve a member of the pencil by suitably choosing $V$ and $B$, this family of test curves must meet any effective divisor, in particular, $D$, non-negatively. So $a \ge 0$.

To handle the $b_i$, we use a remark of Kapranov~\cite{KapranovVeronese} that the set $K$ of maps $[f] \in \Mbar_{0, 0}(\PP^d, d)$ whose image contains a fixed set of $d+2$ linearly general points is disjoint from $G$ (by construction) and may be identified with $\Mbar_{0, d+2}$  (by using the points as the markings), so that points of $\Gamma_i\cap K$ correspond to those of $\Delta_{i+1}$. We can choose $K$ not to lie in $D$ by taking the $(d+2)$-points to lie on a rational normal curve not in $D$ so $K$ must induce an effective class $D_K$ on $\Mbar_{0, d+2}$. But $K$ does not depend on the ordering of the $d+2$ points so $D_K$ is $\Sn$-invariant and the non-negativity of the $b_i$ follows from Lemma~~\ref{effmzeronsym}.
\end{proof}

We also note that the argument about $B$ in the first paragraph of the proof generalizes. If $f$ is a stable map with domain $C$ that is \emph{not} in $G$ or in any of the $\Gamma_i$, then $f(C)$ is an irreducible, non-degenerate curve of degree $d$ in $\PP^d$, hence is a rational normal curve. The translations of $f$ by $\PGL(d+1)$ will thus be the locus of all rational normal curves, which is dense in $\Mbar_{0, 0}(\PP^d, d)$. Thus, we get the following, see \cite{CoskunHarrisStarrEffective}*{Lemma 1.8}:

\begin{corollary}\label{movingcurvesonmoobardd} If $B$ is any reduced, irreducible curve in $\Mbar_{0, 0}(\PP^d, d)$ not lying in $G$ or any of the $\Gamma_i$, then $B$ is a moving curve.
\end{corollary}

We will next look at sharpenings and extensions of these results.

\subsection{Effective classes on $\Mbar_{0, 0}(\PP^r, d)$.}\label{moobarrd}

We computed $\eff{\Mbar_{0, 0}(\PP^d, d)}$ above in terms of the classes $G$ and $\Gamma_i$ in Lemma~\ref{effmoobard} above and we now want to discuss the extensions of Coskun, Harris and Starr~\cite{CoskunHarrisStarrEffective} to $\Mbar_{0, 0}(\PP^r, d)$. Since there is no longer any risk of confusion between boundaries in $\Mbar_{0, n}$ and $\Mbar_{0, 0}(\PP^r, d)$, we will now write $\Delta_i$ for the $\Gamma_i$ defined above\footnote{Note that~\cite{CoskunHarrisStarrEffective} uses $D_{\deg}$ and for our $G$ and $\Delta_{i,d-i}$ for our $\Delta_i$.}. We begin by introducing two other important effective classes on $\Mbar_{0, 0}(\PP^r, d)$.

\begin{definition}
\noindent
	\begin{enumerate}
		\item Let $H$ be the locus of maps whose image meets a fixed codimension $2$ linear subspace $L \subset \PP^d$.
		\item Let $\Deltawt$ be the weighted total boundary defined by
		$$\Deltawt = \sum_i \frac{i(d-i)}{d}\Delta_i\,.$$
	\end{enumerate}	
\end{definition}

By a test curve argument (\cite{CoskunHarrisStarrEffective}*{Lemma 2.1}) using the curves $B_k, 1 \le k \le \lfloor \frac{d}{2} \rfloor $ defined as the one-parameter families of maps whose images contain a fixed set of $d+2$ linearly general points, and meet a fixed subspace of dimension $k$ (and, if $k >1$, a second of dimension $d-k$), these classes are related by

\begin{equation}
\label{GHDeltawtrelation}	 2 G =  \frac{(d+1)}{d} H - \Deltawt\,.
\end{equation}

By~\cite{FultonPandharipande}*{Lemma 14}, the general point of each $B_k$ is a map with image a rational normal curve and hence, by Corollary~\ref{movingcurvesonmoobardd}, all the $B_k$ are moving curves. Since, for $k > 1$, $B_k$, and only $B_k$, meets $\Delta_k$, they are independent, and by construction, $\deg_{B_k}(G) = 0$ for all $k$. Hence the $B_k$ are a set of moving curves spanning the null space of $G$ in the cone of curves of $\Mbar_{0, 0}(\PP^r, d)$.

We can use~(\ref{GHDeltawtrelation}) to identify $$u_{d,r}: V_d:= \sspan_{\C}\Bigl\{H, \Delta_i, i=1,\dots \lfloor \frac{d}{2} \rfloor\Bigr\} \to \pic{\Mbar_{0, 0}(\PP^r, d)}$$ and view all the cones $\eff{\Mbar_{0, 0}(\PP^r, d)}$ for a fixed $d$ as living in $V_d$ as well. The next lemma asserts that these cones are nested and stabilize.

\begin{lemma}[\cite{CoskunHarrisStarrEffective}*{Proposition 1.3}] \label{effmoobarrdstability}
The inclusions $$\eff{\Mbar_{0, 0}(\PP^r, d)} \subset \eff{\Mbar_{0, 0}(\PP^{r+1}, d)}$$
hold for all $r\ge 2$, with equality if $ r\ge d$.
\end{lemma}

Informally, maps to the complement $U$ of a point $p \in\PP^{r+1}$ have codimension $r \ge 2$ in $\Mbar_{0, 0}(\PP^{r+1}, d)$ so projection from $p$ induces a map 
$$h :\pic{\Mbar_{0, 0}(\PP^r, d)}\to \pic{\Mbar_{0, 0}(\PP^{r+1}, d)}$$ sending effective divisors to effective divisors and compatible with the identifications $u_{d,r}$. This gives the inclusions. Equality follows for $r \ge d$ by producing inclusions $\eff{\Mbar_{0, 0}(\PP^r, d)} \subset \eff{\Mbar_{0, 0}(\PP^d, d)}$ as follows. If $D \in \eff{\Mbar_{0, 0}(\PP^{r+1}, d)}$, then the map associated to a general point of $D$ has image spanning a $d$-plane  $W \subset \PP^r$ and the pullback of $D$ by any linear isomorphism $j:\PP^d \to W$ is an effective class with the same coordinates in $V_d$. In the sequel, Lemma~\ref{effmoobarrdstability} lets us define new effective classes for a fixed small $r$ and obtain classes for all larger values and, to simplify, we use the same notation for the prototypical class and its pullbacks.

We will refer the reader to \cite{CoskunHarrisStarrEffective}*{\S 3} for other complementary results---in particular, the construction of moving curves dual to the one-dimensional faces of $\eff{\Mbar_{0, 0}(\PP^d, d)}$, either exactly, assuming the Harbourne-Hirschowitz conjecture, or approximately to any desired accuracy without this assumption.

Because we have such an explicit description of $\eff{\Mbar_{0, 0}(\PP^r, d)}$, and likewise, from the sequel \cite{CoskunHarrisStarrAmple} of $\nef{\Mbar_{0, 0}(\PP^r, d)}$, it is possible, at least for small values of $r$ and $d$, to answer more refined questions. In particular, we can attempt to understand the chamber structure of the stable base locus decomposition of $\eff{\Mbar_{0, 0}(\PP^r, d)}$. The case $r=d$ is particularly interesting and we conclude this subsection by describing the results of Chen, Coskun and Crissman~\cites{ChenLogMinimal, ChenCoskunCrissman} for $d=3$ and $d=4$ which reveal interesting relations with other moduli spaces.

To start we need the rosters of additional effective classes exhibited as geometric loci in Table~\ref{divisorcoords}. Then we need to give the coordinates of all these classes in terms of the basis consisting of $H$ and the boundaries. In fact, they are all of the form $aH + b \Delta + \bwt\Deltawt$. The coefficients, also given in the table, summarize cases worked out in \S2 of \cite{ChenCoskunCrissman} (where the coordinates of many tautological classes are also computed) and earlier results in  \cites{DiazHarris, PandharipandeIntersections, PandharipandeCanonical}, all obtained by standard test curve calculations.

\begin{center}
	\begin{table}[htbp] \renewcommand{\arraystretch}{1.3}
	\begin{tabular}{|c c p{1.6in} r r r  | }%
		\hline
		Divisor &  Least $r$  &  Description of general map with smooth domain& $a$  &  $b$ & $\bwt$ \\
		$T$ & r & $f(C)$ is tangent to a fixed hyperplane. & $\frac{d-1}{d}$ & $0$& $1$\\
		$NL$ & 2 & $f(C)$ has a node lying on a fixed line.&$\frac{(d-1)(2d-1)}{2d}$ & $0$ & $-\frac{1}{2}$\\
		$TN$ & 2 & $f(C)$ has tacnode. & $\frac{3(d-1)(d-3)}{d}$& $4$& $d-9$\\
		$TR$ & 2 & $f(C)$ has triple point.& $\frac{(d-1)(d-2)(d-3)}{2d}$& $-1$& $-\frac{d-6}{2}$\\
		$NI$ & 3 & $f(C)$ is not an isomorphism; a generic $f(C)$ is irreducible rational, of degree $d$, with a single node. & $\frac{(d-1)(d-2)}{d}$& $1$& $-\frac{d}{2}$\\
		\hline
	\end{tabular}
		\vspace{6pt}
	   \caption{\label{divisorcoords} Other classes on $\Mbar_{0, 0}(\PP^r, d)$ defined as geometric loci}
	\end{table}
	\vspace{-36pt}
\end{center}

When $d=r=3$, there is considerable collapsing. The loci $TN$ and $TR$ are empty, $NI=T$ and $NL$ coincides with a class called $F$ in~\cite{ChenLogMinimal} and defined as the closure of the locus of maps meeting a fixed plane in two points collinear with a fixed point. Writing $\mbar(\alpha) := \Proj\bigl(\bigoplus_{m \ge 0} H^0\bigl(m(H+\alpha\Delta)\bigr)\bigr)$ for the model of ``slope'' $\alpha$, this leads ---cf.~\cite{ChenLogMinimal}*{Theorem~1.2} to which we refer for further details---to the picture in Figure~\ref{moobar33cone} of the chamber structure of $\eff{\Mbar_{0, 0}(\PP^3, 3)}$. In this figure, each wall is labeled $D:\alpha$ with $D$ a spanning class from the list above and $\alpha$ its slope. Each chamber is labeled with the model, defined below, that arises as $\mbar(\alpha)$ in its interior with brackets (or parentheses) used to indicate of whether this is (or is not) also the model on the corresponding wall.

\begin{figure}[htbp]
    \centering
    \begin{overpic}[scale=.8]{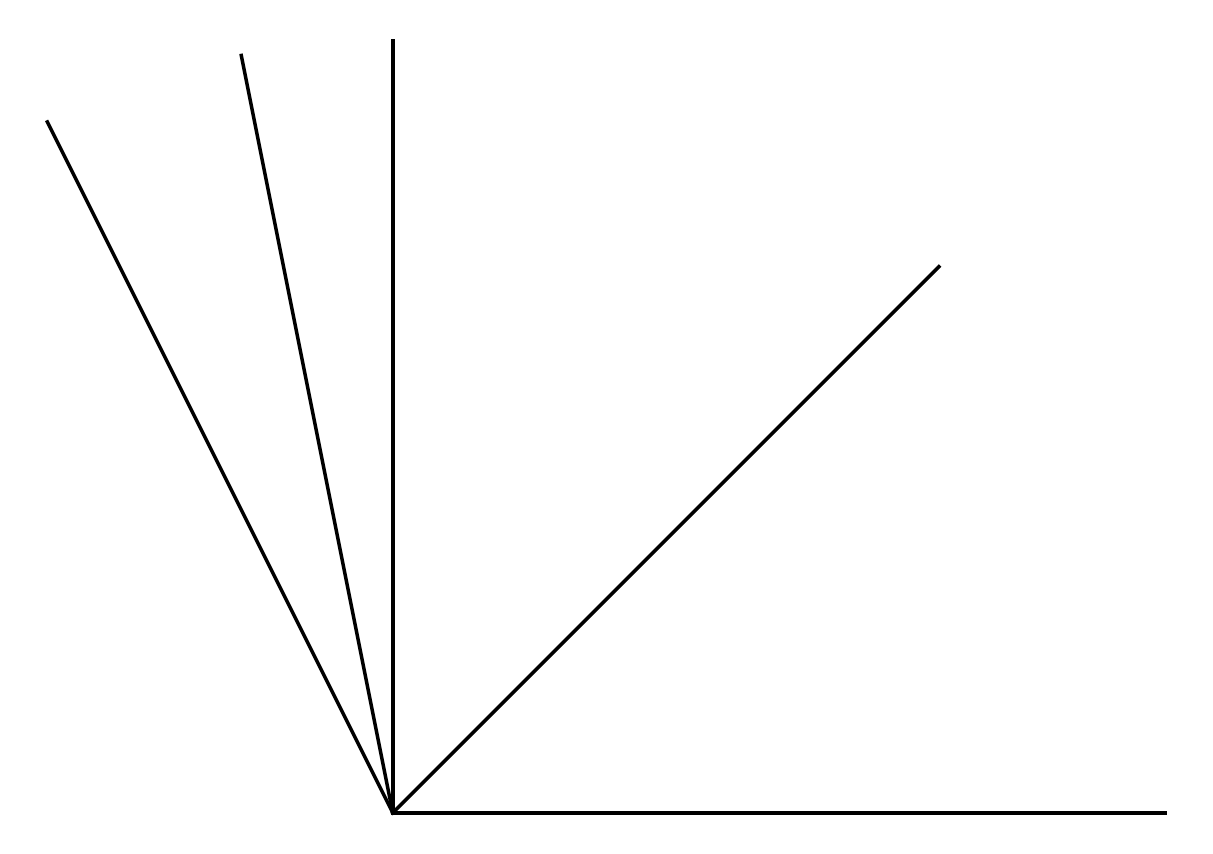}
		\put(00,610){\large$G:-\frac{1}{2}$}
		\put(120,670){\large$NL:-\frac{1}{5}$}
		\put(290,680){\large$H:0$}
		\put(740,490){\large$T:1$}
		\put(970,30){\large$\Delta:\infty$}
		\put(440,490){\large$\bigl(\Mbar_{0, 0}(\PP^3, 3)\bigr)$}
		\put(700,220){\large$\bigl[\moobarMM{\PP^3}{3}{2}\bigr)$}
		\put(240,520){\large$\bigl({\mathcal H}\bigr)$}
		\put(100,490){\large$\bigl({\mathcal H}'\bigr]$}
	\end{overpic}%
	 \vspace{-6pt}
  	\caption{\label{moobar33cone} Chamber structure of $\eff{\Mbar_{0, 0}(\PP^3, 3)}$}
	\vspace{-12pt}
\end{figure}

We will briefly describe the models and wall-crossing maps in the figure, referring to \cites{ChenLogMinimal} for further details. We start at the bottom right with $\Mbar_{0, 0}(\PP^3,3,2)$, which is the space of $2$-stable maps of Musta{\c{t}}{\u{a} and Musta{\c{t}}{\u{a}~\cite{MustataMustata} in which maps whose source has a degree $1$ ``tail'' are replaced by maps of degree $3$ on the ``degree $2$'' component that have a base point at the point of attachment of the tail; the map $\Mbar_{0, 0}(\PP^3, 3)\to\Mbar_{0, 0}(\PP^3,3,2)$ is the contraction of $\Delta$ by forgetting the other end of the tail. The chamber bounded by $T$ and $H$ is, as is shown in~\cite{CoskunHarrisStarrAmple}, the Nef cone of $\Mbar_{0, 0}(\PP^3, 3)$.

The ray through $H$ itself gives a morphism $\phi: \Mbar_{0, 0}(\PP^3, 3) \to \Chow(\PP^3,3)$ by sending a map $f$ to the cycle class of $f(C)$ and this is the right side of a flip that contracts the $8$ dimensional locus $P_3$ where $f$ is a degree $3$ cover of a line and the $9$ dimensional locus $P_{1,2}$ on which $f$ covers a pair of intersecting lines with degrees $1$ and $2$. The Hilbert scheme of twisted cubics contains a $12$-dimensional component ${\mathcal H}$ whose general member is a twisted cubic (and a second $15$-dimensional component---cf.~\cite{PieneSchlessinger}) and the left side of this flip is the cycle map $\psi: {\mathcal H} \to \Chow(\PP^3,3)$ which contracts the $9$-dimensional locus $R \subset {\mathcal H}$ of curves possessing a non-reduced primary component.

Finally, by~\cite{EllingsrudPieneStromme}*{Lemma 2}, \emph{every} point of ${\mathcal H}$ (and not just those coming from twisted cubics) is cut out by a unique net of quadrics, and hence there is a morphism $\rho: {\mathcal H} \to {\mathcal H}' \subset \G(3,10)$ that contracts $G$.

Already, when $r=d=4$, the picture gets substantially more complicated, and not simply because the dimension of this $\eff{\Mbar_{0, 0}(\PP^4, 4)}$ is $3$, In this case, the stable base locus decomposition is no longer completely known. Again we refer to~\cite{ChenCoskunCrissman}*{\S2} for proofs and further details about the claims that follow, and for less complete results about $\Mbar_{0, 0}(\PP^r, d)$ for other values of $r$ and $d$.

\newcommand{\dl}[3]{\put(#1,#2){\hbox{}\kern-1in\hbox to2in{\hfil{#3}\hfil}}}
\begin{figure}[ht]
    \centering
    \begin{overpic}[scale=1]{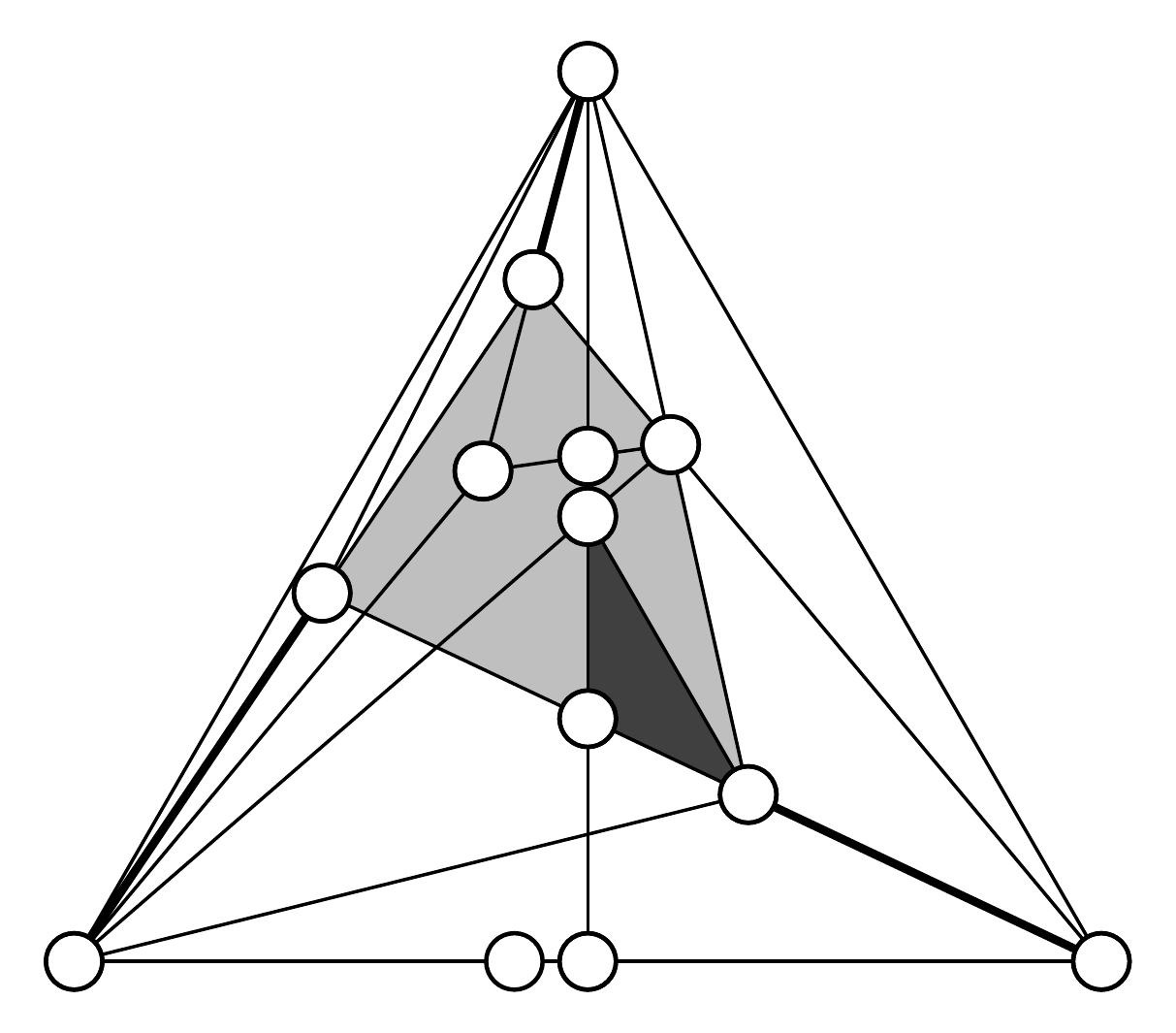}
		\dl{500}{810}{\small G}
		\dl{60}{55}{\small$\Delta'_1$}
		\dl{440}{55}{\small$\Delta$}
		\dl{500}{55}{\small$\Deltawt$}
		\dl{940}{55}{\small$\Delta_2$}
		\dl{500}{485}{\small NL}
		\dl{500}{430}{\small H}
		\dl{500}{260}{\small T}
		\dl{570}{490}{\small TR}
		\dl{638}{195}{\small P}
		\dl{455}{635}{\small NI}
		\dl{410}{470}{\small TN}
		\dl{275}{365}{\small Q}
	\end{overpic}%
	 \vspace{-6pt}
   \caption{\label{moobar44cone} Chamber structure of $\eff{\Mbar_{0, 0}(\PP^4, 4)}$}
\end{figure}

Figure~\ref{moobar44cone} shows a slice in barycentric coordinates in the rays $G$, and the components $\Delta'_{1} := \frac{3}{4}\Delta_1$ and $\Delta_2$ of the weighted boundary $\Deltawt$ (which give a slightly more symmetric picture than using $\Delta_{1}$ and $\Delta_2$). Two extra classes appear as vertices. The first is the class $P = H + \Delta_1+ 4\Delta_2$, which is shown in~\cite{CoskunHarrisStarrAmple}*{Remark 5.1} to be one of the $3$ vertices of $\nef{\Mbar_{0, 0}(\PP^4, 4)}$---the other two are $H$ and $T$. The second is the class $Q = 3H + 3\Delta_1- 2\Delta_2$ defined (up to homothety) as the ray in which the $\Delta_2{-}P{-}T$-plane meets the $\Delta'_1{-}NI$-plane.

In the figure, the white circles label these classes and those in Table~\ref{divisorcoords}. There are no longer any coincidences between these classes but there quite a few coplanarities visible in the figure as collinearities. The lines show the boundaries of chambers in $\eff{\Mbar_{0, 0}(\PP^4, 4)}$ (not necessarily the full set of chambers corresponding to the stable base locus decomposition) in terms of the classes above.  The central triangle shaded in dark gray is, as noted above, $\nef{\Mbar_{0, 0}(\PP^4, 4)}$.

A heavier line segment joins each of the three vertices of $\eff{\Mbar_{0, 0}(\PP^4, 4)}$ to an interior class. The two triangles formed by joining this edge to one of the two other vertices are the locus of divisors whose stable base locus contains the common vertex. For example, the triangle $\Delta_2{\text{-}}P{\text{-}}TR$ is the chamber whose stable base locus contains $\Delta_2$ but \emph{not} $G$ or $\Delta_1$. Together these triangles cover the complement of the light gray quadrilateral which therefore contains the cone of moving divisors; equality would follow if one could produce a locus with class $Q$ and with no divisorial base locus, as the defining descriptions in Table~\ref{divisorcoords} provide for the other vertices.

\subsection{The combinatorial extremal rays of Castravet and Tevelev}\label{mzeronbar}

We now focus on the space $\Mbar_{0, n}$ itself. Since the boundary divisors $\Delta_I$ are an effective set of generators of $\pic{\Mbar_{0, n}}$, a natural question---rendered even more tempting by Corollary~\ref{effmzeronsym}---is whether they generate $\eff{\Mbar_{0, n}}$. This is trivial for $n=3$ and $n=4$ when $\Mbar_{0, n}$ is respectively, a point and $\PP^1$, and easy for $n=5$.

Let us recall the argument from~\cite{HassettTschinkelEffective}*{Proposition 4.1} in the last case. Kapranov's construction~\cite{KapranovVeronese} (or the weighted variant in~\cite{HassettWeighted}) exhibits $\mijbar{0}{5}$, with the $\thst{5}{th}$ marked point distinguished, as the blowup of $\PP^2$ in $4$ general points $p_i$. Denote by $L$ the class of a general line, by $E_i$ the $\thst{i}{th}$ exceptional divisor, and by $E$ the sum of the $E_i$. We can then identify $\Delta_{\{i, 5\}}$ with $E_i$ and $\Delta_{\{i, j\}}$ with $L-E_i-E_j$ (these are the proper transform of the lines through $p_i$ and $p_j$, the other $(-1)$-curves). With these identifications, the $5$ maps $\Mbar_{0, 5}\to \Mbar_{0, 4}$ forgetting the $\thst{i}{th}$ marked point, respectively correspond to the $5$ (semiample) divisors $L-E_i$ and $2L-E$ and the $5$ blow-downs to $\PP^2$ correspond to $2L-E+E_i$ and to $L$.

Brute force calculation shows that, in the vector space spanned by $L$ and the $E_i$, the cone $\mathcal{C}$ spanned by the $10$ boundaries is dual to the cone spanned by these 10 semiample classes. Since the effective cone is the dual of the moving cone of curves and the latter lies inside the dual of the semiample cone, this shows that the boundaries generate $\eff{\Mbar_{0, 5}}$.

However, for any $n\ge 6$, examples due to Keel and Vermeire~\cite{Vermeire} show that there are effective classes $F_\sigma$ that are not effective sums of the boundaries. For $n=6$, fix one of the $15$ partitions $\sigma$ of the marked points into three pairs---say $\sigma= (12)(34)(56)$ which we also view as determining an element in $\Sn$. We can associate to this choice a divisor in two ways. The first is as the fixed locus $F_\sigma$ of the involution of $\Mbar_{0, 6}$ given by $\sigma$. There is a map 
$\phi: \Mbar_{0, 6} \to \Mbar_3$ by identifying the points in each pair to obtain a $3$-nodal irreducible rational curve, and we also obtain $F_\sigma$ as $\phi^*(\Mbar_{3, 2}^1)$ where $\Mbar_{3, 2}^1$ is the closure in $\Mbar_3$ of the hyperelliptic locus in $\M_3$.

To analyze $F_\sigma$, the starting point is again an explicit model of $\Mbar_{0, 6}$ (one starts by blowing up $5$ general points in $\PP^4$, and then blows up the proper transforms of the $10$ lines through 2 of these points). It is again straightforward to write down expressions for the $\Delta_I$ and for fixed loci like $F_\sigma$ as combinations of classes $L$, $E_i$ and $E_{ij}$ defined in analogy with those above (cf. the table on p.79 of~\cite{Vermeire}). Vermeire then gives an essentially diophantine argument with these coefficients to show that $F_\sigma$ is not an effective sum of boundaries. Pulling back $F_\sigma$ by forgetful maps, produces effective classes on $\Mbar_{0, n}$ for any $n \ge 6$ that are not effective sums of boundaries.

In another direction, these examples are known be sufficient to describe only $\eff{\Mbar_{0, 6}}$. Hassett and Tschinkel~\cite{HassettTschinkelEffective}*{Theorem 5.1} prove this by again showing that the dual of the cone generated by these classes lies in the moving cone of curves, for which the use of a tool like \net{Porta} which was convenient for $n=5$ is now essential. Castravet~\cite{CastravetCox} gives another argument that, though quite a bit longer, can be checked by hand, based on showing that the divisors of boundary components and the $F_\sigma$ generate the Cox ring of $\Mbar_{0, 6}$.

At this point, experts were convinced that $\eff{\Mbar_{0, n}}$ probably had many non-boundary extremal rays but there was no clear picture of how they might be classified, indeed there were no new examples, until the breakthrough of Castravet and Tevelev~\cite{CastravetTevelev} in 2009, which provides a recipe for constructing such rays from \emph{irreducible hypertrees} $\Gamma$---combinatorial data whose definition we will give in a moment, along with the related notions of \emph{generic} and \emph{spherical duals}---that they conjecture yields them all.

\begin{theorem}[\cite{CastravetTevelev}*{Theorem 1.5, Lemma 7.8 and Lemma 4.11}] \label{hypertreemaintheorem} Every hypertree $\Gamma$ of order $n$ determines an effective divisor $D_\Gamma \subset \Mbar_{0, n}$.
	\begin{enumerate}
		\item If $\Gamma$ is irreducible, then $D_\Gamma$ is a non-zero, irreducible effective divisor satisfying:
		\begin{enumerate}
			\item $D_\Gamma$ is an extremal ray of $\eff{\Mbar_{0, n}}$ and meets $\M_{0, n}$.
			\item There is a birational contraction $f_\Gamma: \Mbar_{0, n} \dashrightarrow X_\Gamma$ onto a normal projective variety $X_\Gamma$ whose exceptional locus consists of $D_\Gamma$ and components lying in $\Delta$.
			\item If $\Gamma$ and $\Gamma'$ are generic and $D_\Gamma=D_{\Gamma'}$, then $\Gamma$ and $\Gamma'$ are spherical duals.
		\end{enumerate}
		\item The pullback via a forgetful map $\Mbar_{0, n} \to \Mbar_{0, m}$ of the divisor $D_\Gamma$ on $\Mbar_{0, m}$ associated  to any irreducible hypertree of order $m$, which when $n$ is understood we will again (abusively) denote by $D_\Gamma$, also spans an extremal ray of $\eff{\Mbar_{0, n}}$.
		\item If $\Gamma$ is not irreducible, then every irreducible component of $D_\Gamma$---if this locus is nonempty---is pulled back via a forgetful map from the divisor $D_{\Gamma'}\subset \Mbar_{0, m}$ of an irreducible hypertree $\Gamma'$ of order $m < n$.
	\end{enumerate}
\end{theorem}

The table below shows the number $\textrm{IH}(n)$ of irreducible hypertrees of order $n$, up to $\Sn$-equivalence. For $n=5$, this count must be $0$ since all extremal rays are boundaries. For $n=6$, the unique $D_\Gamma$ yields the Keel-Vermeire divisors (cf. Figure~\ref{kvhypertree}).
\begin{center} \renewcommand{\arraystretch}{1.2}
	\begin{tabular}{|c|rrrrrrr|}%
		\hline
		$n$ & $5$ & $6$ &$7$ &$8$ &$9$ &$10$ &$11$\\
		$\textrm{IH}(n)$& $0$ & $1$ &$1$ &$3$ &$11$ &$93$ &$1027$\\
		\hline
	\end{tabular}
\end{center}

As the table indicates, $\textrm{IH}(n)$ grows rapidly with $n$. Empirically, most hypertrees are generic. So the upshot of Theorem~\ref{hypertreemaintheorem} is to provide, as $n$ increases, very large numbers of new extremal effective divisors. As a complement, \cite{CastravetTevelev}*{\S9} gives examples of larger collections of non-generic irreducible hypertrees which give the same ray in $\eff{\Mbar_{0, n}}$.

Based, as far as the authors can tell on the cases $n \le 6$, Castravet and Tevelev propose a very optimistic converse conjecture. We quote from \cite{CastravetTevelev}*{Conjecture~1.1}:

\begin{conjecture}\label{hypertreeconjecture} Every extremal ray of $\eff{\Mbar_{0, n}}$ is either a boundary divisor or the divisor $D_\Gamma$  of an irreducible hypertree $\Gamma$ of order at most $n$.
\end{conjecture}

To our knowledge, more of the activity this has prompted has been devoted to searching for a counterexample than for a proof. Aaron Pixton informs us that he has an example (different from that of~\cite{PixtonExample}) of a divisor $D$ on $\Mbar_{0,12}$ that is effective and non-moving, and that is not equal to any
irreducible hypertree divisor, but whether this divisor lies outside the cone spanned by the  hypertree divisors is not, at the time of writing, clear.

We now turn to defining all the terms used above. We have taken the liberty of introducing a few new terms like triadic (see below) for notions used or referred to often, but not named in~\cite{CastravetTevelev}. It will simplify  our definitions to write $\nnn := \{ 1, 2, \ldots, n\}$ (or any other model set of cardinality $n$).

Recall that a hypergraph $\Gamma$ of order $n$ consists of collection indexed by $j\in \ddd$ of hyperedges $\Gamma_j \subset \nnn$. We say that $\Gamma'$ is a sub-hypergraph of $\Gamma$ if each of the hyperedges $\Gamma'_k$ is a subset of some hyperedge $\Gamma_j$. We start with the notion of \emph{convexity} of a hypergraph.

\begin{definition} \label{convexdefinition}
\noindent
	\begin{enumerate}
\item For any set $ S $ of hyperedges, let $T_S= \bigcup_{j\in S} \Gamma_j$, $\tau_S = |T_S|-2 $ and $\sigma_S = \sum_{j \in S} \bigl( |\Gamma_j| - 2\bigr)$. We call $\Gamma$ \emph{convex} if for all $ S \subset \ddd$, $\tau_S \ge \sigma_S$. Taking $S$ to be a singleton, this implies that every hyperedge contains at least $3$ vertices.  We call $\Gamma$ \emph{strictly convex} if this inequality is strict whenever $2 \le |S| \le (d-1)$.
\item A hypertree is \emph{triadic} if every hyperedge contains exactly $3$ vertices (i.e. $\tau_S= \sigma_S$ for $S$ any hyperedge).  For such a hypertree, convexity simply says that any set $S$ of hyperedges contains at least $|S|+2$ vertices.
	\end{enumerate}
\end{definition}

Now we turn to the notions of \emph{hypertree} and of \emph{irreducibility}.
\begin{definition} \label{hypertreedefinition}
\noindent
	\begin{enumerate}
	\item  A \emph{hypertree} $\Gamma$ of order $n$ is a hypergraph satisfying:
		\begin{enumerate}
			\item Every vertex lies on at least $2$ hyperedges.
			\item (Convexity) $\Gamma$ is convex. In particular, every hyperedge contains at least $3$ vertices.
			\item (Normalization) $\tau_\Gamma = \sigma_\Gamma$; that is, $n-2 = \sum_{j\in \ddd} \bigl(|\Gamma_j|-2\bigr)$.
		\end{enumerate}
	\item A hypertree is \emph{irreducible} if it is strictly convex.
	\end{enumerate}
\end{definition}

Empirically, most hypertrees are triadic in which case the normalization condition simply says that $n=d+2$ and irreducibility says that any proper subset of $e\ge 2$ hyperedges contains at least $e+3$ vertices. The use of the term ``tree'' is motivated by the observation that a dyadic hypergraph (i.e. an ordinary graph) is a tree exactly when $n=d+1$

Now we turn to the notion of \emph{genericity}.
\begin{definition} \label{capacitydefinition}
 We let $\conv(\Gamma)$ denote the set of all convex sub-hypergraphs of $\Gamma$ and define the \emph{capacity} of $\Gamma$ by
$$\capac(\Gamma) := \max_{\Gamma' \in \conv(\Gamma)} \sigma_{\Gamma'}\,.$$
\end{definition}

\begin{definition} \label{genericdefinition}
\noindent
	\begin{enumerate}
\item A triple $T$ of vertices that do not lie on any hyperedge of $\Gamma$ but such that any two do lie on a hyperedge is called a \emph{wheel}\footnote{Here the term \emph{triangle} might better capture the intuition.}  of $\Gamma$.
\item If $\Gamma$ is a triadic hypertree of order $n$ and $T$ is a triple of vertices that is not a hyperedge, we can form a new triadic hypertree $\Gamma_T$ of order $(n-2)$ by identifying the vertices in $T$ and deleting any hyperedges containing $2$ of these vertices.
\item An irreducible triadic hypertree is \emph{generic} if, whenever $T$ is a triple that is neither a hyperedge nor a wheel, we have $\capac(\Gamma_T) = n-4$.
\end{enumerate}
\end{definition}

An important source of generic triadic examples is provided by triangulations of the $2$-sphere in which each vertex has even valence, or equivalently whose triangles can be $2$-colored, say black and white, so that each edge has one side of each color.

\begin{definition} \label{sphericaldefinition}
\noindent
	\begin{enumerate}
\item For any $2$-colorable spherical triangulation, the  triangles of each color form the set of hyperedges of a triadic hypertree---called a \emph{spherical} hypertree---on the full set of vertices, and we say this pair of hypertrees are \emph{spherical duals} (or, in~\cite{CastravetTevelev}, the black and white hypertrees of an even triangulation of the sphere).
\item Given a distinguished triangle in each of two \emph{spherical} hypertrees, we may form their \emph{connected sum} by choosing colorings which make one triangle white and the other black, deleting these two triangles, and glueing along the exposed edges.
\end{enumerate}
\end{definition}

\begin{lemma}\label{sphericalirreducibility} A spherical hypertree is irreducible unless it is a connected sum.	
\end{lemma}

\begin{figure}[ht]
    \centering
    \begin{overpic}[scale=0.9]{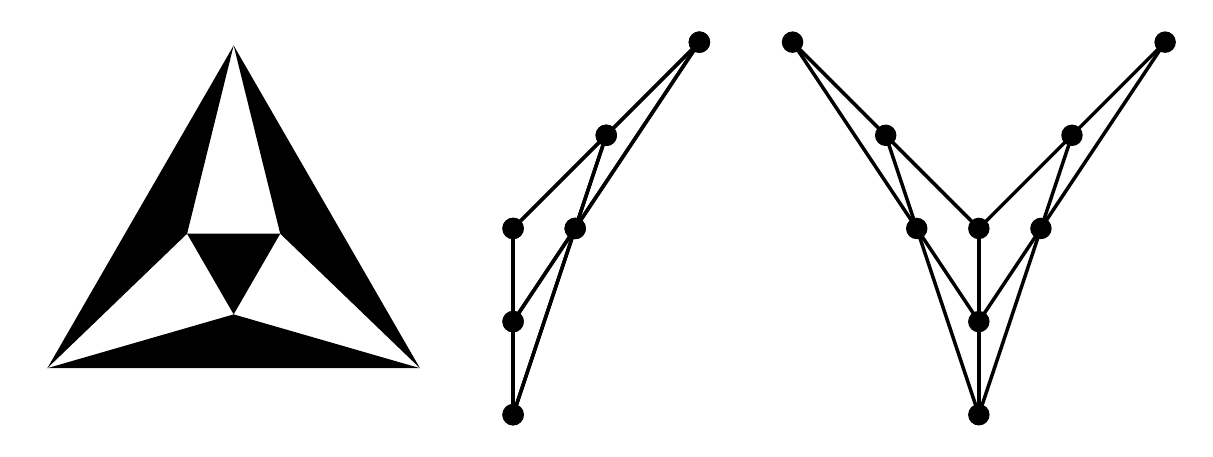}
	\end{overpic}%
 	\vskip-18pt
   \caption{\label{kvhypertree} The Keel-Vermeire hypertree}
\end{figure}

At the left of Figure~\ref{kvhypertree}, we show the simplest $2$-colorable spherical triangulation for which both the spherical duals give the irreducible triadic hypertree $\Gamma_{\textrm{KV}}$ of order $6$ shown in the center. The divisor $D_{\Gamma_{\textrm{KV}}}$ of this hypertree is the Keel-Vermeire divisor. On the right, we show the hypertree given by taking the connected sum of this black and white spherical duals. This is not irreducible by taking $S$ to be the set of hyperedges on the left or right side of the picture.

Proving Theorem~\ref{hypertreemaintheorem} involves delicate combinatorial and geometric arguments that are far too involved to give here. All we will attempt to do is to sketch the main steps of the argument. Castravet and Tevelev~~\cite{CastravetTevelev} also prove many other complementary results that we will not even cite here.

A key motivating idea, though one whose proof comes rather late in the development, is that every irreducible hypertree has a \emph{planar realization}, by which we mean an injection of $f_\Gamma: \nnn \to \PP^2$ so that the set of lines in the plane containing $3$ or more points of the image of $f_\Gamma$ is exactly the set of hyperedges of $\Gamma$. If so, and $\pi_p$ is the projection to $\PP^1$ from a point $p$ not on any of the lines through at least $2$ of the points in this image, then the composition $\pi_p\circ f_\Gamma$ defines a set of $n$ marked points on $\PP^1$ and hence a point $[f_\Gamma, p] \in \Mbar_{0, n}$. A typical realization and projection for the Keel-Vermeire curve are shown in Figure~\ref{kvphypertree}.

\begin{figure}[ht]
    \centering
    \begin{overpic}[scale=0.7]{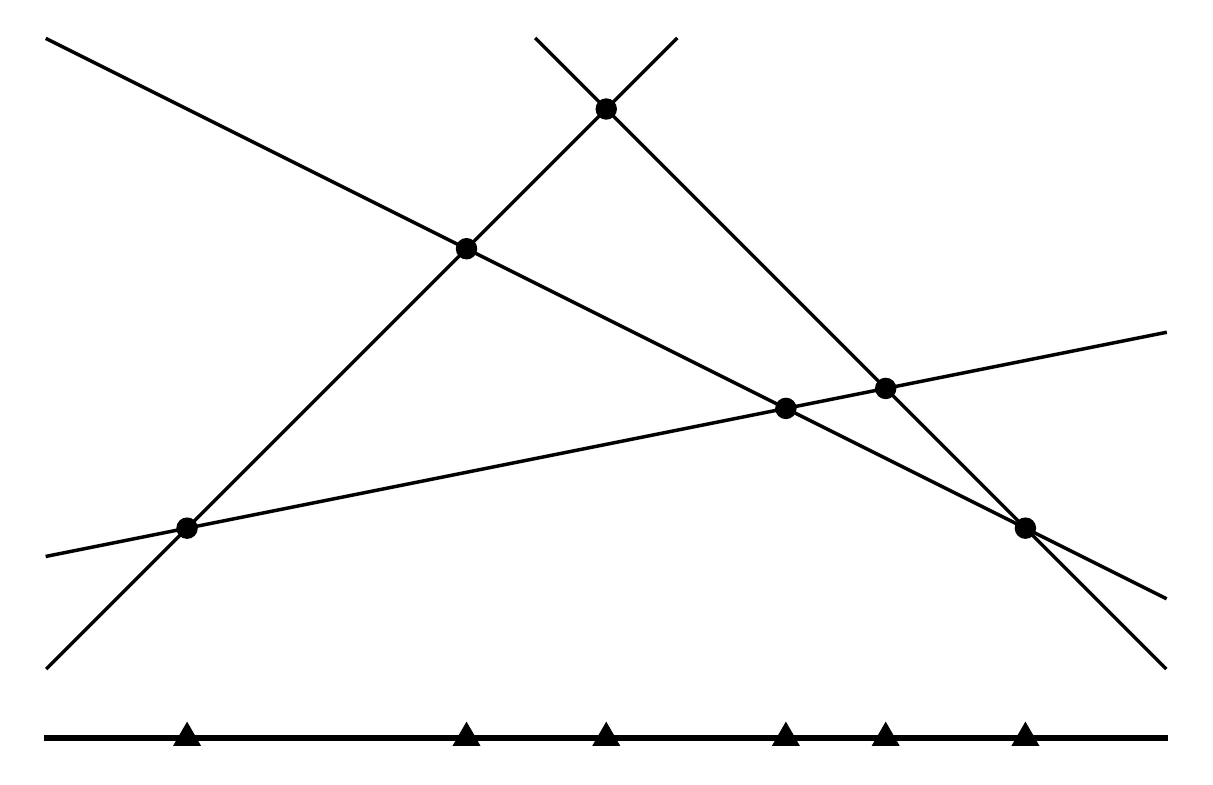}
	\end{overpic}%
   \caption{\label{kvphypertree} Planar realization and projection of the Keel-Vermeire hypertree}
\end{figure}

The closure of the locus of all such points is defined to be $D_\Gamma$ and, by the time the non-emptiness of this locus can be established, the fact that it is an  irreducible divisor has been established via a second description. Both planarity of hypertrees and irreducibility of the loci $D_\Gamma$ are obtained as byproducts of a study of hypertree curves and associated Brill-Noether loci. We will sketch the ideas in the simpler case when $\Gamma$ is triadic, simply hinting at the complications for general $\Gamma$.

The \emph{hypertree curve} $\Sigma_\gamma$ associated to a triadic hypertree $\Gamma$ is obtained by taking a copy of a $3$-pointed $\PP^1$ for each hyperedge of $\Gamma$ and gluing all the points corresponding to the vertex $i$ to a single point $p_i$ as a scheme-theoretic pushout (i.e. so that the branches look locally like the coordinate axes in an affine space of dimension equal to the valence of the vertex). Note that  $\Sigma_\gamma$ has genus $g = n-3= d-1$ and the Picard scheme $\picone$ of line bundles of degree $1$ on each component is (non-canonically) isomorphic to $(\G_m)^g$. For a hypertree whose vertices all have valence $2$, this curve is already stable (as in Figure~\ref{kvphypertree}). In general, to get a stable model, it is necessary to replace each vertex of valence $v \ge 3$ by a $v$-pointed copy of $\PP^1$ glued to the coincident components at its marked points, and to avoid adding moduli, to fix the choice of this curve in some arbitrary way.

Now the connection to Brill-Noether theory enters. For a general smooth curve $\Sigma$ of genus $g\ge 2$, there is a birational morphism $\nu: G^1_{g+1}\to  W^1_{g+1} \simeq \picdeg{g+1}{\Sigma}$ sending a pencil of divisors of degree $g+1$ to its linear equivalence class, and whose exceptional divisor $E$ lies over the codimension $3$ locus $W^2_{g+1}$ of line bundles with  $h^0(L)\ge 3$. The general pencil $D$ in $G^1_{g+1}$ and in $E$  is globally generated, so the general pencil $D$ in $E$ can be obtained as the composition of the map to $\PP^2$ associated to $\nu(E)$ with projection from a general point of $\PP^2$.

The idea unifying all the steps above in~\cite{CastravetTevelev} is to extend this picture to the genus $0$ hypertree curves $\Sigma_\Gamma$ above. Sticking to the simpler case of triadic $\Gamma$ with all vertices of valence $2$, a linear system on $\Sigma_\Gamma$ is \emph{admissible} if it is globally generated and sends the singularities to distinct points and an invertible sheaf is admissible if its complete linear series is. Define $\picone$ to be the set of admissible line bundles having degree $1$ on each component, define the Brill-Noether locus $W^r \subset \picone$ to be the locus of admissible line bundles with $h^0(\Sigma, L) \ge r+1$, and define the locus $G^r$ to be the pre-image of $W^r$ under the natural forgetful map $\nu$ from pencils to line bundles. Again, there are extra complications if the hypertree is not triadic (because then the hypertree curves have moduli), or if there are vertices of higher valence (in which case, sheaves in $\picone$ are required to have degree $0$ on the components inserted as each such vertex).

The main line of argument of~\cite{CastravetTevelev} may then be sketched as follows. Theorem 2.4 identifies $\M_{0, n}$ with $G^1$ and shows that $\nu$ is birational with exceptional locus $G^2$ (and compactifies this picture when $\Gamma$ is not triadic). After an interlude in \S3 devoted to computing the dimensions of images of maps generalizing this compactification to hypergraphs that are not necessarily convex, Theorem 4.2 shows that the divisor $D_\Gamma$ obtained by taking the closure of $G^2$ in $\Mbar_{0, n}$ is non-empty and irreducible and partially computes its class; a by product is the characterization of the components of $D_\Gamma$ (Lemma~4.11) when $\Gamma$ is not irreducible. Section~5 is another interlude proving Gieseker stability (with respect to the dualizing sheaf) of invertible sheaves in the (generalized) $\picone$ which is then applied to complete the construction of the birational contraction of Theorem~\ref{hypertreemaintheorem} (cf. also {Theorem 1.10}). The reconciliation of the descriptions of $D_\Gamma$ as the closure of the locus of plane realizations (in particular, the existence of such realizations) and as the closure of $G^2$ is carried out in \S6 (Theorem 6.2). Finally, that the divisor $D_\Gamma$ of a generic, irreducible $\Gamma$ determines it (up to spherical duals) is proved in \S7.

\bibrefspread

\end{document}